\documentclass{article}
\usepackage{fontenc}
\usepackage{amsmath}
\usepackage{amssymb}
\usepackage{amsfonts}
\usepackage{amsthm}

\newtheorem{thm}{Theorem}[section]
\newtheorem{prop}[thm]{Proposition}
\newtheorem{cor}[thm]{Corollary}

\newtheorem{lem}[thm]{Lemma}

\newtheorem*{resnn}{Main Result}

\begin{document}

\title{Regularized Pairings of Meromorphic Modular Forms and Theta Lifts}

\author{Shaul Zemel\thanks{This work was carried out while I was working at the Technische Universit\"{a}t Darmstadt in Germany, and was partially supported by DFG grant BR-2163/4-1.}}

\maketitle

\section*{Introduction}

Given a positive discriminant $\delta$, we define $\mathcal{Q}_{\delta}$ to be
the set of integral binary quadratic forms of discriminant $\delta$. Given an
integer $k$, the paper \cite{[Za]} introduced the function
\[f_{k,\delta}(z)=\sum_{Q\in\mathcal{Q}_{\delta}}\frac{1}{Q(z,1)^{k}},\] and
proves it to be a cusp form of weight $2k$. Consequently, the paper \cite{[KZ]}
shows that this modular form is the image, under the Shimura lift, of the
$\delta$th Poincar\'{e} series of weight $k+\frac{1}{2}$ for $\Gamma_{0}(4)$
(this is equivalent to the assertion, appearing in that reference, that
$f_{k,\delta}(z)$ is the $\delta$th ``Fourier coefficient'' in the expansion of
the holomorphic kernel for the Shimura--Shintani lift with respect to the
``weight $k+\frac{1}{2}$ variable'' $\tau$). This determines their pairing with
any cusp form of weight $2k$ via the Petersson inner product.

On the other side, \cite{[BK]} considers similar functions arising from
quadratic forms with \emph{negative} discriminant $D$. These are meromorphic
modular forms, again of weight $2k$, which decrease like cusp forms towards
infinity. They also define a regularized Petersson inner product for
meromorphic modular forms, and evaluate the pairing of these functions
$f_{k,D}$ with other meromorphic modular forms of weight $2k$. Note that in both
references only modular forms with respect to $SL_{2}(\mathbb{Z})$ (or
congruence subgroups of low level) are considered.

\smallskip

The purpose of this paper is twofold. First, we show that the meromorphic
modular forms arising from negative discriminants are also lifts of certain
Poincar\'{e} series. Indeed, given a positive integer $m$, \cite{[Ze3]} combines
a theta lift (which is essentially a generalized Shimura lift) with weight
raising operators on both sides to obtain a lift from weakly holomorphic
modular forms (or harmonic weak Maa\ss\ forms) of weight $\frac{1}{2}-m$ to
meromorphic modular forms of weight $2k=2m+2$. We then prove our
\begin{resnn}
Given $r<0$ and $\beta \in L^{*}/L$, let
$F_{m,r,\beta}^{L}\big(\tau,\frac{3}{4}+\frac{m}{2}\big)$ be the harmonic weak
Maa\ss\ form of weight $\frac{1}{2}-m$ and representation $\rho_{L}$ having
principal part $q^{r}\big(e_{\beta}+(-1)^{m}e_{-\beta}\big)$. Applying the lift from \cite{[Ze3]} to $F_{m,r,\beta}^{L}$ produces roughly the modular form
$f_{m+1,D}$.
\end{resnn}
Here $L$ is a specific lattice that is related to integral binary quadratic
forms. Now, by changing the lattice $L$, we can generalize the definition of
$f_{m+1,D}$ to modular forms with respect to various other Fuchsian groups, and
show that they have similar properties. Examples of such groups, which are of arithmetic interest, arise, e.g., from embeddings of indefinite rational quaternion algebras into $M_{2}(\mathbb{R})$ (see Section 1 of \cite{[Ze3]} for more on these groups).

The second goal of this paper is to use this presentation as the
$\delta_{2m}$-image of a theta lift in order to simplify the evaluation of the
pairing appearing in \cite{[BK]}. Moreover, this method immediately generalizes
the assertions from \cite{[BK]} to meromorphic modular forms with respect to
more general Fuchsian groups. In fact, as the theta lift from \cite{[Ze3]}
admits generalizations to modular forms on higher-dimensional Shimura varieties,
this opens a way to investigate whether appropriate meromorphic Hilbert or
Siegel modular forms have similar properties. However, most parts of this paper
restrict themselves to the 1-dimensional case.

\smallskip

The paper is divided into 11 sections. Section \ref{PoinSer} presents the
Poincar\'{e} series of \cite{[Bru]}, with some of their useful properties.
Section \ref{ThetaLift} introduces the theta lifts of \cite{[B]}, \cite{[Bru]},
and \cite{[Ze3]}. Section \ref{UnfPoin} evaluates the theta lift of
Poincar\'{e} series explicitly (in any dimension), while Section \ref{Dim1}
gives the details of the special case of dimension 1. Section \ref{Expatw}
presents the natural coordinate for expanding modular forms around points in
the upper half-plane, while Section \ref{RegPair} uses this coordinate to give
the details of the regularized pairing of \cite{[BK]}. Section \ref{UnfPhi} then
writes the pairing with (our equivalent of) $f_{m+1,D}$ in a form which is
convenient for its evaluation. The additional formulae required in the case
where $\Gamma$ has cusps are given in Section \ref{Cuspz} and evaluated in
Section \ref{Cuspeval}. Section \ref{PolEval} then produces the final expression
for the pairing. Finally, Section \ref{QuadForms} links our functions, in a
special case, to those from \cite{[BK]}, and describes briefly some connections
to other works.

I would like to express my gratitude to K. Bringmann for sharing the details of
\cite{[BK]} with me. I am also thankful to J. Bruinier for his suggestion to
consider the lifts of the Poincar\'{e} series from \cite{[Bru]}, as well as for
many intriguing discussions.

\section{Weight Raising Operators and Poincar\'{e} Series \label{PoinSer}}

In this Section we describe Poincar\'{e} series producing weak Maa\ss\ forms with arbitrary representations of the integral metaplectic group. The description follows Chapter 1 of \cite{[Bru]} rather closely, with the representation being more general. The applications below will, however, use only Weil representations.

Given complex numbers $\nu$ and $\mu$ with $\Re\mu>0$, one defines the
\emph{Whittaker function} $M_{\nu,\mu}$ to be the solution of the Whittaker
differential equation
\[M_{\nu,\mu}''(t)+\bigg(-\frac{1}{4}+\frac{\nu}{t}+\frac{1-4\mu^{2}}{4t^{2}}
\bigg)M_{\nu,\mu}=0\] that satisfies $M_{\nu,\mu}(t) \sim t^{\mu+\frac{1}{2}}$ as $t\to0^{+}$. For $k\in\frac{1}{2}\mathbb{Z}$ and $s\in\mathbb{C}$ with $\Re s>1$ we define, following Section 1.3 of \cite{[Bru]}, the function
\[\mathcal{M}_{k,s}(t)=t^{-\frac{k}{2}}M_{-\frac{k}{2},s-\frac{1}{2}}(t).\] Let $\mathcal{H}$ be the upper half-plane $\{\tau=x+iy\in\mathbb{C}|y>0\}$. For any variable $\xi$ we shorthand $\frac{\partial}{\partial\xi}$ to $\partial_{\xi}$. Hence
\[\partial_{\tau}=\frac{\partial_{x}-i\partial_{y}}{2}\quad\mathrm{and}\quad\partial_{\overline{\tau}}=\frac{\partial_{x}+i\partial_{y}}{2},\] and we define the \emph{weight raising operator}, the \emph{weight lowering operator}, and the \emph{weight $k$ Laplacian} to be \[\delta_{k}=\partial_{\tau}+\frac{k}{2iy},\quad y^{2}\partial_{\overline{\tau}},\quad\mathrm{and}\quad\Delta_{k}=4\delta_{k-2} \cdot y^{2}\partial_{\overline{\tau}}\] respectively. Note that eigenvalues of eigenfunctions are conventionally taken with respect to $-\Delta_{k}$.

We introduce the useful shorthand $\mathbf{e}(z)=e^{2\pi iz}$ for any complex
number $z$. Given $0>r\in\mathbb{Q}$, one proves
\begin{prop}
The function taking $\tau\in\mathcal{H}$ to
$\mathbf{e}(rx)\mathcal{M}_{k,s}(4\pi|r|y)$ is a weight $k$ eigenfunction of
eigenvalue $\frac{k(k-2)}{4}+s(1-s)$. \label{eigen}
\end{prop}

The function $\mathcal{M}_{k,s}(4\pi|r|y)$ grows like $(4\pi|r|y)^{s-\frac{k}{2}}$ as $y\to0$. The other eigenfunction having the same eigenvalue as in Proposition \ref{eigen}, which is based on the Whittaker $W$-function, grows like $y^{1-s-\frac{k}{2}}$, i.e., faster, since we assume $\Re s>1$. Therefore an eigenfunction of eigenvalue $\frac{k^{2}-2k}{4}+s(1-s)$ growing as $o\big(y^{1-s-\frac{k}{2}}\big)$ as $y\to0$ is a multiple of the function from Proposition \ref{eigen}. Now, the commutation relation between the weight changing operators and the corresponding Laplacians show that given a weight $k$ eigenfunction $F$ on $\mathcal{H}$ of eigenvalue $\frac{k(k-2)}{4}+s(1-s)$, the weight $k+2$ function $\delta_{k}F$ has eigenvalue $\frac{(k+2)k}{4}+s(1-s)$. Applying this to the function from Proposition \ref{eigen} and observing the growth of its $\delta_{k}$-images as $y\to0$ establishes
\begin{prop}
Applying $\frac{1}{2\pi i}\delta_{k}$ to
$\tau\mapsto\mathbf{e}(rx)\mathcal{M}_{k,s}(4\pi|r|y)$ yields the function
\[\tau\mapsto-|r|\bigg(s+\frac{k}{2}\bigg)\mathbf{e}(rx)\mathcal{M}_{k+2,s}
(4\pi|r|y).\] \label{deltakPSgen}
\end{prop}

An element of the double cover $Mp_{2}(\mathbb{Z})$ of $SL_{2}(\mathbb{Z})$ is a pair consisting of a matrix $A=\binom{a\ \ b}{c\ \ d}$ from $SL_{2}(\mathbb{Z})$ and a holomorphic function on $\mathcal{H}$ whose square is the factor of automorphy $j(A,\tau)=c\tau+d$. It is generated by the two elements \[T=\bigg(\binom{1\ \ 1}{0\ \ 1},1\bigg)\quad\mathrm{and}\quad S=\bigg(\binom{0\ \ -1}{1\ \ \ \
0},\sqrt{\tau}\in\mathcal{H}\bigg).\] These elements satisfy the relation
$S^{2}=(ST)^{3}=Z=(-I,i)$, and $Z$ generates the center of $Mp_{2}(\mathbb{Z})$,
which is cyclic of order 4. Let $\rho:Mp_{2}(\mathbb{Z}) \to \mathcal{U}(V)$ be
a finite-dimensional unitary representation of $Mp_{2}(\mathbb{Z})$ factoring
through a finite quotient, let $k\in\frac{1}{2}\mathbb{Z}$ be a weight, and let
$r\in\mathbb{Q}$ be a negative number. Choose an element $\omega \in V$ that is
an eigenvector of both $\rho(Z)$ and $\rho(T)$, with eigenvalues $i^{-2k}$ and
$\mathbf{e}(r)$ respectively. Then any function of the sort
$\tau\mapsto\mathbf{e}(rx)M(y)\omega$ is invariant under the slash operators
\[f[A]_{k,\rho}(\tau)=\rho(A)^{-1}j(A,\tau)^{-k}f(A\tau)\] for $A=T$ and $A=Z$
(using the metaplectic data for half-integral weights). Now, for any $A \in
Mp_{2}(\mathbb{Z})$ and $\tau\in\mathcal{H}$ we have $|\mathbf{e}(r\Re
A\tau)|=1$ and \[\big|\mathcal{M}_{k,s}(4\pi|r|\Im A\tau)j(A,\tau)^{-k}\big|=(4\pi|r|y)^{-\frac{k}{2}}\bigg|M_{-\frac{k}{2},s-\frac{1}{2}}
\bigg(\frac{4\pi|r|y}{|j(A,\tau)|^{2}}\bigg)\bigg|.\] As $\rho$ factors through a finite quotient, Proposition \ref{eigen} and the behavior of $M_{-\frac{k}{2},s-\frac{1}{2}}$ for small positive values of the argument imply, for $\Re s>1$ (as in the remark following Definition 1.8 of \cite{[Bru]}), the following
\begin{prop}
The poincar\'{e} series
\[F_{k,r}^{\rho,\omega}(\tau,s)=\frac{1}{4\Gamma(2s)}\sum_{A\in\langle T \rangle \backslash Mp_{2}(\mathbb{Z})}\Big(\mathrm{e}(rx)\mathcal{M}_{k,s}(4\pi|r|y)\omega\Big)[A]_{k,\rho}(\tau)\] converges locally uniformly on $\mathcal{H}$ to a modular form of weight $k$ and representation $\rho$ with respect to $Mp_{2}(\mathbb{Z})$ that is an eigenfunction with eigenvalue $\frac{k(k-2)}{2}+s(1-s)$. \label{Poincare}
\end{prop}
Here and throughout, $\Gamma(\xi)$ stands for the value at $\xi$ of the
classical gamma function.

The convergence in Proposition \ref{Poincare} and the fact that the parameter $s$ is not changed in Proposition \ref{deltakPSgen} combine to give
\begin{cor}
The equality \[\frac{1}{2\pi
i}\delta_{k}F_{k,r}^{\rho,\omega}(\tau,s)=-|r|\bigg(s+\frac{k}{2}\bigg)F_{k+2,r}
^{\rho,\omega}(\tau,s)\] holds for any $\tau\in\mathcal{H}$ and $s\in\mathbb{C}$
with $\Re s>1$. \label{deltakPoin}
\end{cor}

Note that if $k$ is negative then the value $s=1-\frac{k}{2}>1$ yields the
eigenvalue 0 in Proposition \ref{Poincare}. The resulting modular form is thus a harmonic weak Maa\ss\ form. For these forms we have a richer structure, which is investigated in detail in Section 3 of \cite{[BF]}. Moreover, Proposition 1.10 of \cite{[Bru]} shows that $F_{k,r}^{\rho,\omega}\big(\tau,1-\frac{k}{2}\big)=q^{r}\omega+O(1)$ as $y\to\infty$ (both references consider only Weil representations, but the results extend immediately to our more general representations of $Mp_{2}(\mathbb{Z})$), so that in particular the image of $F_{k,r}^{\rho,\omega}\big(\tau,1-\frac{k}{2}\big)$ under the operator
$\xi_{k}=y^{k}\overline{\partial_{\overline{\tau}}}$ of \cite{[BF]} lies in the
space of cusp forms of weight $2-k$ and representation $\overline{\rho}$. Here and throughout we use the classical notation $q=\mathbf{e}(\tau)$. Now, any  principal part of a harmonic weak Maa\ss\ form with cuspidal $\xi_{k}$-image is a finite sum of such principal parts $q^{r}\omega$ with $r<0$ and $\omega \in V$ such that $\rho(T)\omega=\mathbf{e}(r)\omega$ and $\rho(Z)\omega=i^{-2k}\omega$, and for a representation factoring through a finite quotient a harmonic weak Maa\ss\ form of negative weight is determined by its principal part (this is not necessarily true for representations not factoring through a finite quotient---see, e.g., \cite{[Ze2]}). This proves, as in Proposition 1.12 of \cite{[Bru]}, the following
\begin{prop}
The space of harmonic weak Maa\ss\ forms of weight $k$ and representation $\rho$ with cuspidal $\xi_{k}$-images is spanned, for any $k<0$ and a representation $\rho$ factoring through a finite quotient, by the Poincar\'{e} series from Proposition \ref{Poincare} with $s=1-\frac{k}{2}$.
\label{hwMfPoin}
\end{prop}

We remark that all the statements of this section hold if we replace
$Mp_{2}(\mathbb{Z})$ by any of its subgroups of finite index, a fact which can
be easily seen either by averaging or by using induced representations.

\section{Theta Lifts \label{ThetaLift}}

Let $L$ be an \emph{even lattice} of signature $(b_{+},b_{-})$. This means a free Abelian group of finite rank with a non-degenerate bilinear form $L \times
L\to\mathbb{Z}$, such that $\lambda^{2}=(\lambda,\lambda)$ is even for every $\lambda \in L$, and such that the extension of the bilinear form to the real vector space $L_{\mathbb{R}}=L\otimes\mathbb{R}$ has signature $(b_{+},b_{-})$. The group $L^{*}=Hom(L,\mathbb{Z})$ is embedded into $L_{\mathbb{R}}$, containing $L$ with finite index. The \emph{discriminant group} $D_{L}=L^{*}/L$ carries a $\mathbb{Q}/\mathbb{Z}$-valued quadratic form
$\gamma\mapsto\frac{\gamma^{2}}{2}$. If we assume that $b_{+}=2$ then the
Grassmannian $G(L_{\mathbb{R}})$, which is the set of decompositions of
$L_{\mathbb{R}}$ into the orthogonal direct sum of a positive definite space
$v_{+}$ and a negative definite space $v_{-}$, carries the structure of a
complex manifold. Indeed, fixing an isotropic vector $z \in L_{\mathbb{R}}$
yields the Lorentzian space $K_{\mathbb{R}}=z^{\perp}/\mathbb{R}z$, in which the
choice of $z$ and of a continuous orientation on the positive definite part
$v_{+}$ determines one cone $C$ of positive norm vectors in $K_{\mathbb{R}}$ to
be the \emph{positive cone}. Choosing $\zeta \in L_{\mathbb{R}}$ with
$(z,\zeta)=1$ identifies $K_{\mathbb{R}}$ with the subspace
$\{z,\zeta\}^{\perp}$ of $L_{\mathbb{R}}$, and maps $G(L_{\mathbb{R}})$
homeomorphically onto the tube domain $K_{\mathbb{R}}+iC$. The inverse map takes
$Z \in K_{\mathbb{R}}+iC$ to the element of $G(L_{\mathbb{R}})$ in which $v_{+}$
is spanned by the real and imaginary parts of the norm 0 vector
\[Z_{V,Z}=Z+\zeta-\frac{Z^{2}+\zeta^{2}}{2}z \in L_{\mathbb{C}}.\] For more
details on this construction see Section 13 of \cite{[B]}, Section 3.2 of
\cite{[Bru]}, or Section 2.2 of \cite{[Ze3]}.

The connected component $SO^{+}(L_{\mathbb{R}})$ of $O(L_{\mathbb{R}})$ operates
on $G(L_{\mathbb{R}})$, and therefore also on $K_{\mathbb{R}}+iC$. The action of
an element $M$ of the latter group sends, for any $Z \in K_{\mathbb{R}}+iC$,
the norm 0 vector $Z_{V,Z}$ to some multiple of $Z_{V,MZ}$, and the multiplier
$J(M,Z)$ defines a factor of automorphy for this action. We call a function
$\Phi:K_{\mathbb{R}}+iC\to\mathbb{C}$ an \emph{automorphic form of weight $m$}
with respect to a discrete subgroup $\Gamma$ of $SO^{+}(L_{\mathbb{R}})$ if it
satisfies the usual functional equations \[\Phi(MZ)=J(M,Z)^{m}\Phi(Z)\] for any
$M\in\Gamma$ and $Z \in K_{\mathbb{R}}+iC$. The natural group to take for
$\Gamma$ is the intersection $Aut(L) \cap SO^{+}(L_{\mathbb{R}})$, or the kernel
of the canonical map from the latter group into $Aut(D_{L})$, called the
\emph{discriminant kernel} of $L$.

Given a general even lattice $L$, the group $Mp_{2}(\mathbb{Z})$ admits a (Weil)
representation $\rho_{L}$ on the space $\mathbb{C}[D_{L}]$. Every canonical
basis vector $e_{\gamma}$ with $\gamma \in D_{L}$ is an eigenvector of
$\rho_{L}(T)$ with eigenvalue $\mathbf{e}\big(\frac{\gamma^{2}}{2}\big)$, and
$\rho_{L}(S)$ operates, up to a constant, as the Fourier transform:
\[\rho_{L}(S)e_{\gamma}=\frac{\mathbf{e}\big(\frac{b_{-}-b_{+}}{8}\big)}{\sqrt{
|D_{L}|}}\sum_{\delta \in
D_{L}}\mathbf{e}\big(-(\gamma,\delta)\big)e_{\delta}.\] The action of
$\rho_{L}(Z)$ sends $e_{\gamma}$ to $i^{b_{-}-b_{+}}e_{-\gamma}$. The vectors on which $\rho_{L}(Z)$ operates as $i^{-2k}$ are thus as follows. In case $2k \equiv b_{+}-b_{-}\pmod4$ they are spanned by the combinations $e_{\beta}+e_{-\beta}$ (including just $2e_{\beta}$ in case $2\beta=0$ in $D_{L}$). If $2k-2 \equiv b_{+}-b_{-}\pmod4$ then they are generated by the differences $e_{\beta}-e_{-\beta}$ for $\beta \in D_{L}$ of order not dividing 2, and there are no such vectors if $2k-1 \equiv b_{+}-b_{-}\pmod2$. Note that all these generators are eigenvectors of $T$, allowing us to use any of them in order to define Poincar\'{e} series as in Proposition \ref{Poincare}. In case $b_{+}=2$, $k=1-\frac{b_{-}}{2}+m$ for some $m\in\mathbb{Z}$, and $\omega=e_{\beta}+(-1)^{m}e_{-\beta}$, we denote the function $F_{k,r}^{\rho,\omega}$ by $F_{m,r,\beta}^{L}$. For more on the representation $\rho_{L}$ see \cite{[Ze1]}, as well as the references cited there.

Back in the case $b_{+}=2$, given $\lambda \in L^{*}$ and $Z=X+iY \in
K_{\mathbb{R}}+iC$, we denote $\lambda_{\pm}$ the projection of $\lambda$ onto
the $v_{\pm}$-part according to the element of $G(L_{\mathbb{R}})$ corresponding
to $Z$. Then \cite{[Ze3]} considers, for some $0<m\in\mathbb{N}$, the theta
function \[\Theta_{L,m,m,0}(\tau,Z)=\sum_{\lambda \in
L^{*}}\frac{(\lambda,Z_{V,Z})^{m}}{(Y^{2})^{m}}\mathbf{e}\bigg(\tau\frac{
\lambda_{+}^{2}}{2}+\overline{\tau}\frac{\lambda_{-}^{2}}{2}\bigg)e_{\lambda+L}
.\] Here $\tau\in\mathcal{H}$, $Z \in K_{\mathbb{R}}+iC$, and
$\Theta_{L,m,m,0}(\tau,Z)\in\mathbb{C}[D_{L}]$, where the coefficient of
$e_{\beta}$ for $\beta \in D_{L}$ is
\[\theta_{L+\beta,m,m,0}(\tau,Z)=\sum_{\lambda \in
L+\beta}\frac{(\lambda,Z_{V,Z})^{m}}{(Y^{2})^{m}}\mathbf{e}\bigg(\tau\frac{
\lambda_{+}^{2}}{2}+\overline{\tau}\frac{\lambda_{-}^{2}}{2}\bigg).\] The
properties of this theta function are given in
\begin{thm}
$(i)$ Let $Z$ be a fixed element of $K_{\mathbb{R}}+iC$. The function sending $\tau\in\mathcal{H}$ to $y^{b_{-}/2}\Theta_{L,m,m,0}(\tau,Z)$ is modular of weight $1-\frac{b_{-}}{2}+m$ and representation $\rho_{L}$. $(ii)$ If $\tau$ is fixed then considering the complex conjugate of $\theta_{L+\beta,m,m,0}$, for any $\beta \in D_{L}$, as a function of $Z$, it is an automorphic form of weight $m$ with respect to the discriminant kernel of $L$. \label{theta}
\end{thm}

\begin{proof}
Part $(i)$ is just a special case of Theorem 4.1 of \cite{[B]}. Part $(ii)$
follows, for example, from Proposition 3.3 of \cite{[Ze3]} and the behavior of
these theta functions under complex conjugation. This proves the theorem.
\end{proof}

The main technical object of interest in \cite{[Ze3]} is the \emph{theta lift}
of the image $F$ of a weakly holomorphic modular form (or a harmonic weak
Maa\ss\ form with $\xi$-image a cusp form) $f$ of weight $\kappa=1-\frac{b_{-}}{2}-m$ and representation $\rho_{L}$ under the $m$-fold weight raising operator $\frac{1}{(2\pi i)^{m}}\delta_{\kappa}^{m}$. The function $F$ has weight $1-\frac{b_{-}}{2}+m$ and eigenvalue $-\frac{mb_{-}}{2}$. Moreover, one has
\begin{lem}
Every modular form of weight $1-\frac{b_{-}}{2}+m$ and eigenvalue
$-\frac{mb_{-}}{2}$ is the $\delta_{\kappa}^{m}$-image of a harmonic weak Maa\ss\ form of weight $\kappa=1-\frac{b_{-}}{2}-m$. \label{deltamiso}
\end{lem}

\begin{proof}
We have the map $\delta_{\kappa}^{m}$ from harmonic weak Maa\ss\ forms of weight $\kappa$ to modular form of weight $1-\frac{b_{-}}{2}+m$ and eigenvalue $-\frac{mb_{-}}{2}$. In addition, consider the map $(4y^{2}\partial_{\overline{\tau}})^{m}$ in the other direction. A direct evaluation shows that both compositions yield the corresponding identity maps, multiplied by the scalar $m!\Gamma\big(m+\frac{b_{-}}{2}\big)/\Gamma\big(\frac{b_{-}}{2}\big)$. This
immediately implies the assertion of the lemma.
\end{proof}

The theta lift of $F$ is essentially the Petersson inner product of $F$ with
$y^{b_{-}/2}\Theta_{L,m,m,0}$ as a function of $Z$. However, the resulting
integral does not converge because of the exponential growth of $F$ as
$y\mapsto\infty$, and has to be regularized. \cite{[B]} and \cite{[Bru]}
suggest two ways to do this, both of which are based on carrying out the
integration over the fundamental domain
\[D=\big\{\tau\in\mathcal{H}\big||\tau|\geq1,\ |x|\leq1/2\big\}\] first over
$x$ and then over $y$. More precisely, one defines the truncated fundamental
domain \[D_{R}=\big\{\tau \in D\big|y \leq H\big\}\] for $R\geq1$, which is
compact and on which the integral of a smooth function converges, and considers
the limit
\begin{equation}
\Phi_{L,m,m,0}(Z,F)=\lim_{R\to\infty}\int_{D_{R}}\big\langle
F(\tau),\Theta_{L,m,m,0}(\tau,Z) \big\rangle y^{m+1}\frac{dxdy}{y^{2}}.
\label{lift}
\end{equation}
However, this limit does not always exist. Given $\lambda \in L^{*}$ with
$\lambda^{2}=0$, the sub-Grassmannian \[\lambda^{\perp}=\big\{v \in
G(L_{\mathbb{R}})\big|\lambda \in v_{-}\big\}\] is a complex sub-manifold of
$G(L_{\mathbb{R}})$ of codimension 1. Then \cite{[Bru]} considers, for a
Poincar\'{e} series, the limit from Equation \eqref{lift} as a function of $s$
(this exists for $s$ in some right half plane, provided that $Z$ does not
belong to a specific $\lambda^{\perp}$), meromorphically continues it, and takes
the constant term at the required value of $s$. In fact, in \cite{[Bru]} only
the case $k=1-\frac{b_{-}}{2}$ (with a vector $\omega=e_{\beta}+e_{-\beta}$) is
considered, but the theory works for more general weights. On the other hand,
\cite{[B]} multiplies the integrand in Equation \eqref{lift} by $y^{-\tilde{s}}$
for another variable $\tilde{s}$, obtains a holomorphic function of $\tilde{s}$
in some right half plane, and again uses a meromorphic continuation and takes
the constant term at $\tilde{s}=0$. Since $m>0$ and the term from $\lambda=0$
does not contribute to the theta function, an argument similar to the proof of
Theorem 3.9 of \cite{[Ze3]} or to the proof of Proposition 2.8 of \cite{[Bru]}
(modified to suit our theta function) shows
\begin{prop}
The limit in Equation \eqref{lift} exists wherever $F$ is defined as
$\frac{1}{(2\pi i)^{m}}\delta_{\kappa}^{m}f$ for $f$ harmonic of weight $\kappa=1-\frac{b_{-}}{2}-m$ with $\xi_{\kappa}f$ cuspidal. It also exists wherever $F$ is $F_{m,r,\beta}^{L}(\tau,s)$ with $\Re s>1$ and $Z$ does not belong to any $\lambda^{\perp}$ for $\lambda \in L+\beta$ with $\lambda^{2}=-2m$. Moreover, if $F=\frac{1}{(2\pi i)^{m}}\delta_{\kappa}^{m}f$ for $f$ the harmonic Poincar\'{e} series $F_{-m,r,\beta}^{L}\big(\tau,\frac{1}{2}+\frac{b_{-}}{4}+\frac{m}{2}\big)$, and
$Z$ does not lie on any such $\lambda^{\perp}$, then the regularizations of
\cite{[B]} and \cite{[Bru]} coincide. \label{tworegs}
\end{prop}

Indeed, the pole appearing in Proposition 2.8 of \cite{[Bru]} arises from the
contribution of the term with $\lambda=0$, which vanishes in the case we
consider. Note that by Corollary \ref{deltakPoin}, the modular form $F$ in the
latter assertion of Proposition \ref{tworegs} is $(-1)^{m}m!|r|^{m}F_{+m,r,\beta}^{L}\big(\tau,\frac{1}{2}+\frac{b_{-}}{4}+\frac{m}{2}\big)$, where the second variable can also be written as $1-\frac{\kappa}{2}$. In addition, Propositions \ref{hwMfPoin} and \ref{tworegs} allow us to evaluate the theta lift of any $F=\frac{1}{(2\pi i)^{m}}\delta_{\kappa}^{m}f$ as linear combinations of the regularized integrals in the sense of \cite{[Bru]}. In any case, part $(ii)$ of Theorem \ref{theta} implies that as a functions of $Z$, the function $\Phi_{L,m,m,0}(Z,F)$ is, under any regularization, an automorphic form of weight $m$ with respect to the discriminant kernel of $L$. Moreover, Theorem 3.9 of \cite{[Ze3]} shows that it is an eigenfunction, with eigenvalue $-2mb_{-}$, with respect to (minus) the Laplacian of $G(L_{\mathbb{R}})$ given explicitly in that reference.

\section{Unfolding \label{UnfPoin}}

Let us now evaluate the theta lift
\[(-|r|)^{m}m!\frac{i^{m}}{2}\Phi_{L,m,m,0}\Bigg(Z,F_{m,r,\beta}^{L}\bigg(\cdot,
\frac{1}{2}+\frac{b_{-}}{4}+\frac{m}{2}\bigg)\Bigg),\] which we denote
$\Phi_{m,r,\beta}^{L}(Z)$. Theorem 3.9 of \cite{[Ze3]} describes this function
in terms of a Fourier expansion at a cusp (if cusps exist), and gives its
singularities. However, for our applications it will be more convenient to have
an alternative description, for which we use the unfolding method from Section
2.3 of \cite{[Bru]}.

Let
\[F(a,b,c;t)=\sum_{n=0}^{\infty}\frac{\Gamma(a+n)\Gamma(b+n)\Gamma(c)}{
\Gamma(a)\Gamma(b)\Gamma(c+n)}\cdot\frac{t^{n}}{n!}\] be the Gau\ss\
hypergeometric series, assuming that neither $-a$, nor $-b$, nor $-c$ are
natural numbers. Our theta lift is given in
\begin{thm}
The value of the theta lift
$\Phi_{L,m,m,0}\big(Z,F_{m,r,\beta}^{L}\big(\cdot,s)\big)$ equals the constant
$2(2|r|)^{s-\frac{1}{2}+\frac{b_{-}}{4}-\frac{m}{2}}\Gamma\big(s-\frac{1}{2}
+\frac{b_{-}}{4}+\frac{m}{2}\big)/\Gamma(2s)$ times \[\sum_{\lambda \in
L+\beta,\
\lambda^{2}=2r}\frac{(\lambda,\overline{Z_{V,Z}})^{m}/(Y^{2})^{m}}{(2\pi)^{m}
|\lambda_{-}^{2}|^{s-\frac{1}{2}+\frac{b_{-}}{4}+\frac{m}{2}}}F\bigg(s-\frac{1}{
2}+\frac{b_{-}}{4}+\frac{m}{2},s+\frac{1}{2}-\frac{b_{-}}{4}+\frac{m}{2},
2s;\frac{2|r|}{|\lambda_{-}^{2}|}\bigg),\] wherever $s\in\mathbb{C}$ satisfies
$\Re s>\frac{3}{2}+\frac{b_{-}}{4}+\frac{m}{2}$ and $Z \in K_{\mathbb{R}}+iC$
does not lie on any  $\lambda^{\perp}$ for $\lambda \in L+\beta$ with
$\lambda^{2}=2r$. \label{unflift}
\end{thm}

\begin{proof}
The proof follows Theorem 2.14 of \cite{[Bru]}. The expression from Equation
\eqref{lift} becomes $\frac{1}{4\Gamma(2s)}$ times the limit of
\[\int_{D_{R}}\sum_{A\in\langle T \rangle \backslash
Mp_{2}(\mathbb{Z})}\big\langle\big[\mathrm{e}(rx)\mathcal{M}_{k,s}
(4\pi|r|y)\omega\big][A]_{k,\rho_{L}}(\tau),\Theta_{L,m,m,0}(\tau,Z)\big\rangle
y^{m+1}\frac{dxdy}{y^{2}}\] as $R\to\infty$, where $k=1-\frac{b_{-}}{2}+m$ and
$\omega=e_{\beta}+(-1)^{m}e_{-\beta}$. Part $(i)$ of Theorem \ref{theta} now
shows that for every $A$ the latter integrand can be written as
\[\bigg\langle\frac{\mathrm{e}(r\Re A\tau)\mathcal{M}_{k,s}(4\pi|r|\Im
A\tau)}{j(A,\tau)^{k}}\rho_{L}^{-1}(A)\omega,\frac{\rho_{L}^{-1}(A)\Theta_{L,m,m,0}(A\tau,Z)}{j(A,\tau)^{k}|j(A,\tau)|^{b_{-}}}\bigg\rangle
y^{m+1}\frac{dxdy}{y^{2}}.\] As the side of $\Theta$ is conjugated, the value of $k$ shows that the power of $y$ and the $j(A,\tau)$ factors become just $(\Im
A\tau)^{m+1}$. The fact that $\rho_{L}$ is unitary allows us to make the change
of variables to $A\tau$ (which we choose such that $|\Re A\tau|\leq\frac{1}{2}$), and obtain that $\Phi_{L,m,m,0}\big(Z,F_{m,r,\beta}^{L}\big(\cdot,s)\big)$ equals
$\frac{1}{4\Gamma(2s)}$ times \[\lim_{R\to\infty}\sum_{A\in\langle T \rangle \backslash Mp_{2}(\mathbb{Z})}\int_{AD_{R}}\big\langle\mathbf{e}(rx)\mathcal{M}_{k,s}(4\pi|r|y)\omega,\Theta_{L,m,m,0}(\tau,Z)\big\rangle y^{m+1}\frac{dxdy}{y^{2}}.\]
Since $\theta_{L-\beta,m,m,0}=(-1)^{m}\theta_{L+\beta,m,m,0}$ and the action of
$Mp_{2}(\mathbb{Z})$ on $\mathcal{H}$ factors through $PSL_{2}(\mathbb{Z})$,
this integral becomes just \[\frac{2}{\Gamma(2s)}\lim_{R\to\infty}\sum_{A\in\langle T \rangle \backslash PSL_{2}(\mathbb{Z})}\int_{AD_{R}}\mathbf{e}(rx)\mathcal{M}_{k,s}(4\pi|r|y)\overline{\theta_{L+\beta,m,m,0}(\tau,Z)}y^{m-1}dxdy.\]

Now, the argument proving Lemma 2.13 of \cite{[Bru]} shows that for fixed $Z$,
the theta function $\theta_{L,m,m,0}$ is bounded by a constant times
$y^{-1-m-\frac{b_{-}}{2}}$ as $y\to0^{+}$, uniformly in $x$. The growth condition on $\mathcal{M}_{k,s}$ for $k=1-\frac{b_{-}}{2}+m$ as $y\to0^{+}$ thus shows that our integrand is $O\big(y^{s-\frac{m}{2}-\frac{b_{-}}{4}-\frac{5}{2}}\big)$. Hence if $\Re s>\frac{3}{2}+\frac{b_{-}}{4}+\frac{m}{2}$ then the limit $R\to\infty$ of the sum of terms $A\not\in\langle T \rangle$ becomes just the integral over $\big\{\tau\in\mathcal{H}\big||x|\leq\frac{1}{2},\tau \not\in D\}$. We therefore evaluate $\frac{2}{\Gamma(2s)}$ times the limit of
\[\int_{0}^{R}\!\int_{-1/2}^{1/2}\!\mathbf{e}(rx)\mathcal{M}_{k,s}(4\pi|r|y)\!\sum_{\lambda \in L+\beta}\!\frac{(\lambda,\overline{Z_{V,Z}})^{m}}{(Y^{2})^{m}}\mathbf{e}\bigg(-x\frac{\lambda^{2}}{2}\bigg)e^{-\pi
y(\lambda^{2}-2\lambda_{-}^{2})}y^{m-1}dxdy\] as $R\to\infty$. The integral over $x$ vanishes unless $\lambda^{2}=2r$, for which the exponent becomes $e^{-2\pi
y(r-\lambda_{-}^{2})}=e^{-2\pi y(|\lambda_{-}^{2}|-|r|)}$. Since
$\mathcal{M}_{k,s}(4\pi|r|y)=O(e^{2\pi|r|y})$ as $y\to\infty$ and
$Z\not\in\lambda^{\perp}$, the expression \[\mathcal{M}_{k,s}(4\pi|r|y)e^{-2\pi
y(|\lambda_{-}^{2}|-|r|)}=O(e^{-2\pi y(|\lambda_{-}^{2}|-2|r|)})\] still decays exponentially (as $|\lambda_{-}^{2}|=\lambda_{+}^{2}+2|r|>2|r|$ under our assumption on $Z$). Hence we may just take the upper limit to be $\infty$, and after plugging in the definition of $\mathcal{M}_{k,s}$ we get $2(4\pi|r|)^{-\frac{k}{2}}/\Gamma(2s)$ times
\[\sum_{\lambda \in L+\beta,\ \lambda^{2}=2r}\frac{(\lambda,\overline{Z_{V,Z}})^{m}}{(Y^{2})^{m}}\int_{0}^{\infty}M_{-\frac{k}{2},s-\frac{1}{2}}(4\pi|r|y)e^{-2\pi y(|\lambda_{-}^{2}|-|r|)}y^{m-1-\frac{k}{2}}dy.\] But putting $\alpha=4\pi|r|$, $p=2\pi|\lambda_{-}^{2}|-\frac{\alpha}{2}>\frac{\alpha}{2}$,
$\kappa=-\frac{k}{2}$, $\mu=s-\frac{1}{2}$, and $\nu=m-\frac{k}{2}$ in Equation
(11) on page 215 of \cite{[EMOT2]} shows that the latter integral equals
\[\frac{(4\pi|r|)^{s}\Gamma\big(s+m-\frac{k}{2}\big)}{(2\pi|\lambda_{-}^{2}|)^{s+m-\frac{k}{2}}}
F\bigg(s+m-\frac{k}{2},s+\frac{k}{2},2s;\frac{2|r|}{|\lambda_{-}^{2}|}\bigg).\] After one puts the external coefficient back in, cancels the powers of $2\pi$, and substitutes the value of $k$, this completes the proof of the theorem.
\end{proof}

Let \[B(p,q;T)=\int_{0}^{T}\xi^{p-1}(1-\xi)^{q-1}d\xi,\quad\Re p>0,\quad 0 \leq
T<1\] be the incomplete beta function. Theorem \ref{unflift} now has the
following
\begin{cor}
If $Z \in K_{\mathbb{R}}+iC$ does not lie on any $\lambda^{\perp}$ for $\lambda
\in L+\beta$ with $\lambda^{2}=2r$ then the function $\Phi_{m,r,\beta}^{L}$
attains at $Z$ the value \[(-i)^{m}m!\sum_{\lambda \in L+\beta,\
\lambda^{2}=2r}\frac{(\lambda,\overline{Z_{V,Z}})^{m}}{(4\pi
Y^{2})^{m}}B\bigg(\frac{b_{-}}{2}+m,-m;\frac{2|r|}{|\lambda_{-}^{2}|}\bigg).\]
\label{PhiLmm0}
\end{cor}

\begin{proof}
Recall that $\Phi_{m,r,\beta}^{L}(Z)$ is obtained by substituting
$s=\frac{1}{2}+\frac{b_{-}}{4}+\frac{m}{2}$ in the function from Theorem
\ref{unflift}, and multiplying the result by the constant
$\frac{(-i|r|)^{m}m!}{2}=\frac{(-i)^{m}(2|r|)^{m}m!}{2^{m+1}}$. This value of
$s$ does not lie in the domain considered in Theorem \ref{unflift}, but
substitution is possible due to analytic continuation. It follows that
$\Phi_{m,r,\beta}^{L}(Z)$ equals
\[\frac{(-i)^{m}m!}{\frac{b_{-}}{2}+m}\!\!\!\sum_{\lambda \in L+\beta,\
\lambda^{2}=2r}\!\!\!\frac{(2|r|)^{\frac{b_{-}}{2}+m}(\lambda,\overline{Z_{V,Z}}
)^{m}}{|\lambda_{-}^{2}|^{\frac{b_{-}}{2}+m}(4\pi
Y^{2})^{m}}F\bigg(\frac{b_{-}}{2}+m,1+m,1+\frac{b_{-}}{2}+m;\frac{2|r|}{
|\lambda_{-}^{2}|}\bigg)\] (the denominator $\frac{b_{-}}{2}+m$ arises from the
quotient $\Gamma\big(s-\frac{1}{2}+\frac{b_{-}}{4}+\frac{m}{2}\big)/\Gamma(2s)$
with our value of $s$, due to the classical functional equation of the gamma
function). But a hypergeometric series of the form $F(p,1-q,p+1;T)$ (with $\Re
p>0$) can be written as $\frac{p}{T^{p}}B(p,q;T)$ by the formula appearing in
\cite{[EMOT1]}, Subsection 2.5.4, page 87. When we substitute this in the latter
equation, the two occurrences of $\frac{b_{-}}{2}+m$ and the powers of $2|r|$
and $|\lambda_{-}^{2}|$ cancel out, and we get the desired expression. This
proves the corollary.
\end{proof}

\section{The Case $b_{-}=1$ \label{Dim1}}

We now consider the case of signature $(2,1)$. We may then assume that $L$ is a
lattice in the real quadratic space $M_{2}(\mathbb{R})_{0}$ of traceless
$2\times2$ matrices, in which the norm of a matrix $U$ is $-2\det U$ and the
pairing of $U$ and $V$ is $Tr(UV)$. The action of $SL_{2}(\mathbb{R})$ by
conjugation yields an isomorphism between $PSL_{2}(\mathbb{R})$ and the
connected component of the identity of $O(L_{\mathbb{R}}) \cong O(2,1)$. If we
choose $z$ to be the isotropic vector $\binom{0\ \ 1}{0\ \ 0}$ and $\zeta$ as
$\binom{0\ \ h}{1\ \ 0}$ for some $h\in\mathbb{R}$ (which equals
$\frac{\zeta^{2}}{2}$) then $K_{\mathbb{R}}$ is the one-dimensional space of
matrices of the form $\binom{u\ \ \ \ 0}{0\ \ -u}$ (with norm $2u^{2}$), $C$
consists of such matrices with $u>0$, and $G(L_{\mathbb{R}})$ is isomorphic to
$\mathcal{H}$. For $z=u+iv\in\mathcal{H}$ the isotropic vector $Z_{V,Z}$, which
we denote by $M_{z}$, is $\binom{z\ \ -z^{2}}{1\ \ -z\ }$, its complex conjugate is
$M_{\overline{z}}=\binom{\overline{z}\ \ -\overline{z}^{2}}{1\ \ -\overline{z}\
}$, and the corresponding negative definite space (the orthogonal complement of
the real and imaginary parts of $M_{z}$) is spanned by the norm $-2$ vector
$J_{z}=\frac{1}{v}\binom{u\ \ -|z|^{2}}{1\ \ \ -u\ }$.

The following expressions and evaluations will turn out useful for examining
this case as well as relating it to other references (in particular
\cite{[BK]}):
\begin{lem}
$(i)$ For $\lambda=\binom{b/2\ \ \ c\ \ }{-a\ \ -b/2} \in L^{*}$ one has
\[(\lambda,M_{z})=az^{2}+bz+c,\quad(\lambda,M_{\overline{z}})=a\overline{z}^{2}
+b\overline{z}+c,\quad\mathrm{and}\quad
(\lambda,J_{z})=\frac{a|z|^{2}+bu+c}{v}.\] $(ii)$ The weight raising operators
act via
\[\delta_{2}\frac{(\lambda,M_{\overline{z}})}{v^{2}}=0,\quad\delta_{0}(\lambda,
J_{z})=\frac{i(\lambda,M_{\overline{z}})}{2v^{2}},\quad\mathrm{and}\quad\delta_{
-2}(\lambda,M_{z})=i(\lambda,J_{z}).\] $(iii)$ The action of the weight lowering
operator is by \[v^{2}\partial_{\overline{z}}(\lambda,M_{z})=0,\quad
v^{2}\partial_{\overline{z}}(\lambda,J_{z})=-\frac{i}{2}(\lambda,M_{z}),
\quad\mathrm{and}\quad
v^{2}\partial_{\overline{z}}\frac{(\lambda,M_{\overline{z}})}{v^{2}}=-i(\lambda,
J_{z}).\] $(iv)$ $\lambda^{2}=\frac{b^{2}-4ac}{2}$ and
$\lambda_{-}^{2}=-\frac{(\lambda,J_{z})^{2}}{2}$. \label{b-=1exp}
\end{lem}

\begin{proof}
These are all simple, straightforward calculations, where part $(iv)$ uses also
the fact that $J_{z}^{2}=-2$.
\end{proof}

A vector $\lambda$ of negative norm $2r$ must be of the form
$\pm\sqrt{|r|}J_{w}$ for a unique $w=\sigma+it\in\mathcal{H}$. In this case we
have additional presentations for those appearing in part $(i)$ of Lemma
\ref{b-=1exp}:
\begin{lem}
If $\lambda=-\sqrt{|r|}J_{w}$ then $(\lambda,M_{z})$,
$(\lambda,M_{\overline{z}})$, and $(\lambda,J_{z})$ are
\[\sqrt{|r|}\frac{(z-w)(z-\overline{w})}{t},\quad\sqrt{|r|}\frac{(\overline{z}
-w)(\overline{z}-\overline{w})}{t},\quad\mathrm{and}\quad2\sqrt{|r|}\cosh
d(z,w),\] where $d(z,w)$ is the hyperbolic distance between $z$ and $w$.
\label{b-=1neg}
\end{lem}

We recall that the hyperbolic cosine of the hyperbolic distance between two
points $z$ and $w$ in $\mathcal{H}$ is given by
\begin{equation}
\cosh
d(z,w)=1+\frac{|z-w|^{2}}{2vt}=\frac{|z|^{2}-2u\sigma+|w|^{2}}{2vt}=\frac{
|z-\overline{w}|^{2}}{2vt}-1. \label{coshd}
\end{equation}

\begin{proof}
This follows directly from part $(i)$ of Lemma \ref{b-=1exp}, since the entries
of $\lambda$ are $a=\frac{\sqrt{|r|}}{t}$, $b=-\frac{2\sqrt{|r|}}{t}\sigma$, and
$c=\frac{\sqrt{|r|}}{t}|w|^{2}$.
\end{proof}

Considering elements of negative norm $2r$ in $L$, we define $S_{\beta,r}$ to be
the set of $w\in\mathcal{H}$ such that $-\sqrt{|r|}J_{w} \in L+\beta$. Those
with the opposite sign belong to $S_{-\beta,r}$. We thus obtain the following
expression for the theta lift from Corollary \ref{PhiLmm0} for the case
$b_{-}=1$:

\begin{cor}
If $b_{-}=1$ then the function $\Phi_{m,r,\beta}^{L}$ from Corollary \ref{PhiLmm0} attains at a point $z\in\mathcal{H}$ the value
\[|r|^{m/2}m!\sum_{w \in S_{\beta,r}\cup(-1)^{m}S_{-\beta,r}}\frac{(\overline{z}-w)^{m}(\overline{z}-\overline{w})^{m}}{(8\pi
itv^{2})^{m}}B\bigg(m+\frac{1}{2},-m;\frac{1}{\cosh^{2}d(z,w)}\bigg),\] provided that $z \not\in S_{\beta,r} \cup S_{-\beta,r}$. \label{b-=1lift}
\end{cor}
Here and throughout, the union with $(-1)^{m}S_{-\beta,r}$ means that the terms arising from elements of $S_{-\beta,r}$ must be multiplied by $(-1)^{m}$.
\begin{proof}
This follows directly from Lemma \ref{b-=1neg} and part $(iv)$ of Lemma
\ref{b-=1exp}.
\end{proof}
Note that if $2\beta=0$ in $D_{L}$ then $S_{\beta,r}=S_{-\beta,r}$. In this case
we have, for even $m$, just twice the sum over $S_{\beta,r}$, while for odd $m$
the two sums cancel. This is in correspondence with the fact that the lifted
Poincar\'{e} series $F_{m,r,\beta}^{L}$ vanishes for odd $m$, as $e_{\beta}$ is
an eigenvector of $\rho_{L}(Z)$ but with the wrong eigenvalue.

\smallskip

It will be more convenient to analyze expressions involving the incomplete beta
function as in Corollary \ref{b-=1lift} in terms of the following
\begin{lem}
$(i)$ For $T>1$ we may write
\[B_{m}(T)=B\bigg(m+\frac{1}{2},-m;\frac{1}{T^{2}}\bigg)\quad\mathrm{as}
\quad\int_{T}^{\infty}\frac{2d\xi}{(\xi^{2}-1)^{m+1}}.\] $(ii)$ The function
$B_{m}(T)$ satisfies the estimate $B_{m}(T)=O\big(\frac{1}{T^{m+1}}\big)$ as
$T\to\infty$. \label{Bm1/T2}
\end{lem}

\begin{proof}
Differentiating $B\big(m+\frac{1}{2},-m;\frac{1}{T^{2}}\big)$ gives the
derivative
\[\xi^{m-1/2}(1-\xi)^{1+m}\Big|_{\xi=1/T^{2}}\cdot\frac{-2}{T^{3}}
=\frac{-2}{(T^{2}-1)^{m+1}}\] of the asserted function. As both functions tend
to 0 as $T\to\infty$ (since the incomplete beta function vanishes at 0 by
definition), this proves part $(i)$. For part $(ii)$ we write the integrand
$\frac{2}{(\xi^{2}-1)^{m+1}}$ from part $(i)$ (with $\xi>1$) as
\[\frac{2/\xi^{2m+2}}{(1-1/\xi^{2})^{m+1}}=\frac{2}{\xi^{2m+2}}\sum_{h=0}^{
\infty}\binom{-m-1}{h}\frac{(-1)^{h}}{\xi^{2h}}=\sum_{h=0}^{\infty}\binom{m+h}{h
}\frac{2}{\xi^{2m+2+2h}}.\] Integration term by term now yields the desired
assertion. This proves the lemma.
\end{proof}

\smallskip

As the factor of automorphy on $G(L_{\mathbb{R}})\cong\mathcal{H}$ is
$j(M,z)^{2}$, an automorphic form of weight $m$ is a modular form of weight
$2m$. In addition, the weight $m$ Laplacian on
$G(L_{\mathbb{R}})\cong\mathcal{H}$ is just the usual weight $2m$ Laplacian on
$\mathcal{H}$. As in Theorem 3.10 of \cite{[Ze3]}, the fact that our theta lift
has eigenvalue $-2m$ means that its $\delta_{2m}$-image must be meromorphic. A
formula for this $\delta_{2m}$-image is now given in
\begin{thm}
In the case $b_{-}=1$ we have
\[\frac{1}{2\pi
i}\delta_{2m}\Phi_{m,r,\beta}^{L}(z)=\frac{2i|r|^{m/2}m!}{(-1)^{m}(4\pi)^{m+1}}
\sum_{w \in
S_{\beta,r}\cup(-1)^{m}S_{-\beta,r}}\frac{(2it)^{m+1}}{(z-w)^{m+1}(z-\overline{w
})^{m+1}}.\] \label{delta2mPhi}
\end{thm}

\begin{proof}
We apply $\delta_{2m}$ to the expression given in Corollary \ref{b-=1lift}.
Recall the Leibnitz rule $\delta_{k+l}(fg)=\delta_{k}f \cdot g+f\delta_{l}g$
for weight raising operators, and observe that the expression multiplying the
incomplete beta function is some constant times the $m$th power of
$\frac{(\lambda,M_{\overline{z}})}{v^{2}}$ for $\lambda=-\sqrt{|r|}J_{w}$ (Lemma
\ref{b-=1neg} again). Part $(ii)$ of Lemma \ref{b-=1exp} hence shows that it
suffices to let $\delta_{0}=\partial_{z}$ operate on the incomplete beta
function. Write the latter function as $B_{m}\big(\cosh d(z,w)\big)$, and recall
from Lemma \ref{b-=1neg} that the latter argument is
$\frac{(\lambda,J_{w})}{2\sqrt{|r|}}$. Part $(i)$ of Lemma \ref{Bm1/T2} and part
$(ii)$ of Lemma \ref{b-=1exp} now imply that $\partial_{z}B_{m}\big(\cosh
d(z,w)\big)$ equals
\[\frac{-2}{(\cosh^{2}d(z,w)-1)^{m+1}}\cdot\frac{i(\overline{z}-w)(\overline{z}
-\overline{w})}{4tv^{2}}=\frac{-i(2t)^{2m+1}v^{2m}(\overline{z}-w)(\overline{z}
-\overline{w})}{|z-w|^{2m+2}|z-\overline{w}|^{2m+2}},\] where in the latter
equality we decomposed $\cosh^{2}d(z,w)-1$ as the product of $\cosh d(z,w)-1$
and $\cosh d(z,w)+1$ and used Equation \eqref{coshd} for each of the
multipliers. Dividing by $2\pi i$, plugging in the remaining parts of the
expression from Corollary \ref{b-=1lift}, and canceling the powers of
$\overline{z}-w$,$\overline{z}-\overline{w}$, $t$, $v$, 2, and $i$ now yields
the desired expression. This proves the theorem.
\end{proof}

Note that the expression from Theorem \ref{delta2mPhi} yields precisely the
pole predicted by Theorem 3.10 of \cite{[Ze3]} in such a point $w$: The
parameter $\beta$ from that reference is chosen here to be 1, and the only
non-zero coefficients of the principal part of $f$ are
$c_{-\sqrt{|r|}J_{w},r}=1$ and $c_{\sqrt{|r|}J_{w},r}=(-1)^{m}$. We thus indeed
have the equality \[\sum_{\alpha J_{w} \in L^{*}}\alpha^{m}c_{\alpha J_{w},-\alpha^{2}}=(-\sqrt{|r|})^{m}+(-1)^{m}\sqrt{|r|}^{m}=2(-1)^{m}|r|^{\frac{m}{2}}.\]

\smallskip

The harmonic Poincar\'{e} series
$f=F_{-m,r,\beta}^{L}\big(\tau,\frac{1}{2}+\frac{b_{-}}{4}+\frac{m}{2}\big)$
(with any $b_{-}$), and its image under $\frac{1}{(2\pi i)^{m}}$, have real
Fourier coefficients. This can be seen either as in \cite{[Bru]}, or by
investigating the function $\overline{f(-\overline{\tau})}$ (which is modular
with the same representation $\rho_{L}$ by the properties of the latter
representation). Theorem 2.8 of \cite{[Ze4]} then shows that applying the $m$th
power of the weight lowering operator $L^{(b_{-})}$ defined in that reference to $\Phi_{m,r,\beta}^{L}$ yields
$m!\Gamma\big(m+\frac{b_{-}}{2}\big)(Y^{2})^{m}/\Gamma\big(\frac{b_{-}}{2}\big)$
times the complex conjugate of $\Phi_{m,r,\beta}^{L}$ (note that the latter
function already contains the coefficient $\frac{i^{m}}{2}$ considered in that
reference). Moreover, this function is also harmonic with respect to the
Laplacian defined in \cite{[Ze3]}. For $b_{-}=1$, where
$L^{(1)}=(2v^{2}\partial_{\overline{z}})^{2}$, this means that
\[(2v^{2}\partial_{\overline{z}})^{2m}\Phi_{m,r,\beta}^{L}=m!\frac{\Gamma\big(m+\frac{1}{2}\big)}{\Gamma\big(\frac{1}{2}\big)}
(2v^{2})^{m}\overline{\Phi_{m,r,\beta}^{L}}=\frac{(2m)!}{2^{m}}v^{2m}\overline{\Phi_{m,r,\beta}^{L}}.\] Hence the function $\frac{(4v^{2}\partial_{\overline{z}})^{2m}\Phi_{m,r,\beta}^{L}}{(2m)!}$, a
modular form of weight $-2m$ which is harmonic outside its singularities, also
equals $(2v^{2})^{m}\overline{\Phi_{m,r,\beta}^{L}}$. There are two operators
taking such functions to meromorphic modular forms. One is
$\delta_{-2m}^{2m+1}$, which is just the holomorphic operator
$\partial_{z}^{2m+1}$. The other one is the operator $\xi_{-2m}$ of \cite{[BF]}.
The images under these operators are are given in the following
\begin{prop}
The images of the function
$\frac{(4v^{2}\partial_{\overline{z}})^{2m}\Phi_{m,r,\beta}^{L}}{(2m)!}$ under
$\xi_{-2m}$ and under $\partial_{z}^{2m+1}$ are both multiples of $\frac{1}{2\pi
i}\delta_{2m}\Phi_{m,r,\beta}^{L}$, the coefficients being $2^{m+1}\pi i$ and
$2\pi i(2m)!$ respectively. \label{D2m+1xi}
\end{prop}

\begin{proof}
Using the second description of this function we get
\[\xi_{-2m}(2v^{2})^{m}\overline{\Phi_{m,r,\beta}^{L}}=2^{m}v^{-2m}\overline{
\partial_{\overline{z}}v^{2m}\overline{\Phi_{m,r,\beta}^{L}}}=2^{m}\delta_{2m}
\Phi_{m,r,\beta}^{L},\] which establishes the first assertion. For the second
assertion we recall that an application of $\delta_{k}$ to a function of the
form $(4v^{2}\partial_{z})G$ for $G$ a modular form of weight $k+2$ and
eigenvalue $\lambda$ gives just $\Delta_{k}G=-\lambda G$. We apply this $2m$
times, and divide by $(2m)!$. In total, applying $\delta_{-2m}^{2m}$ to
$\frac{(4v^{2}\partial_{\overline{z}})^{2m}\Phi_{m,r,\beta}^{L}}{(2m)!}$ gives
us $\Phi_{m,r,\beta}^{L}$ again, but multiplied by the constant
$\frac{1}{(2m)!}\prod_{p=1}^{2m}p(2m+1-p)=(2m)!$. This completes the proof of
the proposition.
\end{proof}

\section{Expansion of Modular Forms \label{Expatw}}

Given $\varepsilon>0$, we define $\mathcal{B}_{\varepsilon}$ to be the disc
$\big\{\zeta\in\mathbb{C}\big||\zeta|<\varepsilon\big\}$ of radius $\varepsilon$
around 0. In particular, $\mathcal{B}_{1}$ is the unit disc. Fourier expansions
of modular forms use the coordinate $q=\mathbf{e}(\tau)$ to map $\mathcal{H}$
onto $\mathcal{B}_{1}\setminus\{0\}$. One may consider them as ``Taylor expansions'' at the cusp $\infty$. We shall now introduce a useful coordinate for expanding a modular form at a point in $\mathcal{H}$.

Given $w=\sigma+it\in\mathcal{H}$, we consider the matrix
$A_{w}=\frac{1}{\sqrt{2it}}\binom{1\ \ -w}{1\ \ -\overline{w}} \in
SL_{2}(\mathbb{C})$ (where $\sqrt{i}=\frac{1+i}{\sqrt{2}}$). Its useful
properties are given in the following
\begin{lem}
$(i)$ The M\"{o}bius action of $A_{w}$ takes $z\in\mathcal{H}$ to
$\zeta=\frac{z-w}{z-\overline{w}}$, which lies in $\mathcal{B}_{1}$, and in
particular $A_{w}(w)=0$. $(ii)$ The derivative $A_{w}'(z)$ equals
$\frac{2it}{(z-\overline{w})^{2}}$. $(iii)$ The inverse matrix $A_{w}^{-1}$ is
$\frac{1}{\sqrt{2it}}\binom{-\overline{w}\ \ w}{-1\ \ 1}$, and it sends
$\zeta\in\mathcal{B}_{1}$ to
$z=\frac{w-\overline{w}\zeta}{1-\zeta}\in\mathcal{H}$. $(iv)$ Under the change
of variables $\zeta=A_{w}(z)$ the expressions $v=\Im z$, $z-\overline{w}$,
$z-w$, and $dz$ become
\[\frac{t(1-|\zeta|^{2})}{|1-\zeta|^{2}},\quad\frac{2it}{1-\zeta},\quad\frac{
2it\zeta}{1-\zeta},\quad\mathrm{and}\quad\frac{2it}{(1-\zeta)^{2}}d\zeta\]
respectively. \label{Awprop}
\end{lem}

\begin{proof}
All of these assertions follow from direct and simple calculations.
\end{proof}

The relations between the expressions from Lemma \ref{Awprop} and the action of
$SL_{2}(\mathbb{R})$ on $\mathcal{H}$ are given in
\begin{lem}
Let $\gamma \in SL_{2}(\mathbb{R})$ and points $z$ and $w$ in $\mathcal{H}$ be
given, and let $j=j(\gamma,w)$. Then the following equalities hold: \[(i)\quad
A_{w}\gamma^{-1}=\binom{j/|j|\ \ \ \ 0\ \  }{\ \ 0\ \ \ \
\overline{j}/|j|}A_{\gamma w}.\qquad(ii)\quad A_{\gamma w}(\gamma
z)=\frac{j(\gamma,\overline{w})}{j(\gamma,w)}A_{w}(z).\] \label{Awgamma}
\end{lem}

\begin{proof}
These can also be easily verified directly.
\end{proof}

Part $(ii)$ of Lemma \ref{Awgamma} immediately yields the following simple
\begin{cor}
The equality $|A_{\gamma w}(\gamma z)|=|A_{w}(z)|$ holds for any $\gamma \in
SL_{2}(\mathbb{R})$ and $z$ and $w$ from $\mathcal{H}$. \label{absAwzinv}
\end{cor}

Corollary \ref{absAwzinv} is useful for proving that certain regularized
integrals are well-defined---see Proposition \ref{pairwd} below.

\smallskip

When we wish to expand a meromorphic modular form of weight $2m+2$ with respect
to some group $\Gamma$ around a point $w\in\mathcal{H}$, we let $\zeta=A_{w}(z)$
and $z=A_{w}^{-1}(\zeta)$ and write
\begin{equation}
g(z)=g[A_{w}^{-1}]_{2m+2}(\zeta)j(A_{w}^{-1},\zeta)^{2m+2}=\frac{(1-\zeta)^{2m+2
}}{(2it)^{m+1}}\sum_{n\gg-\infty}a_{n}(w)\zeta^{n}, \label{expwzeta}
\end{equation}
or, in terms of $z$,
\begin{equation}
g(z)=g[A_{w}^{-1}]_{2m+2}(\zeta)j(A_{w},z)^{-2m-2}=\frac{(2it)^{m+1}}{
(z-\overline{w})^{2m+2}}\sum_{n\gg-\infty}a_{n}(w)A_{w}(z)^{n}. \label{expwz}
\end{equation}
For the coefficients $a_{n}(w)$ we can now prove, as in Lemma 8.1 of
\cite{[BK]}, the following
\begin{prop}
The function taking $w\in\mathcal{H}$ to $\frac{a_{n}(w)}{t^{m+1+n}}$ satisfies
the functional equations of a modular form of weight $2m+2+2n$ with respect to
$\Gamma$. \label{modcoeff}
\end{prop}

\begin{proof}
Fix $\gamma\in\Gamma$ and $z\in\mathcal{H}$, and consider the equality $g(\gamma
z)=j(\gamma,z)^{2m+2}g(z)$. We expand the right hand side around $w$ as in
Equation \eqref{expwz}, while for the left hand side we take the expansion
around $\gamma w$ as in the same equation. Using part $(i)$ of Lemma
\ref{Awgamma} for $j(A_{\gamma w},\gamma z)$ and part $(ii)$ of that Lemma for
$A_{\gamma w}(\gamma z)$ we obtain, after eliminating the common leading
coefficient, the equality \[\sum_{n}a_{n}(w)A_{w}(z)^{n}=\sum_{n}a_{n}(\gamma
w)A_{w}(z)^{n}\frac{j(\gamma,\overline{w})^{m+1+n}}{j(\gamma,w)^{m+1+n}}.\] As
the latter equality is an equality of Laurent series in $A_{w}(z)$, we can
compare the coefficients, from which the assertion now easily follows by the
modularity property of $w \mapsto t$. This proves the proposition.
\end{proof}

In spite of Proposition \ref{modcoeff}, we do not call the coefficients $a_{n}$
modular forms, since they are, in general, not continuous. For example, if $g$
has a pole of order $-n$ for some negative integer $n$, then $a_{n}$ attains a
non-zero value at the pole of $g$, but not around it. Proposition \ref{modcoeff}
also has the following
\begin{cor}
The function $w \mapsto a_{-m-1}(w)$ is well-defined on $Y_{\Gamma}=\Gamma\backslash\mathcal{H}$. \label{a-m-1}
\end{cor}

\begin{proof}
One way to see this is as a special case of Proposition \ref{modcoeff}.
Alternatively, and more conceptually, the differential form
\[\bigg(\frac{i(J_{w},M_{z})}{2}\bigg)^{m}g(z)dz=\frac{(z-w)^{m}(z-\overline{w}
)^{m}}{(2it)^{m}}g(z)dz\] is a well-defined meromorphic differential on a
neighborhood of $\Gamma w \in Y_{\Gamma}$ (assuming that both $z$ and $w$ lie in the same pre-image of this neighborhood in $\mathcal{H}$). If we expand $g$ as in Equation \eqref{expwzeta} and apply part $(iv)$ of Lemma \ref{Awprop}, then we find that in terms of $\zeta=A_{w}(z)$ this differential form becomes just
$\sum_{n}a_{n}(w)\zeta^{m+n}d\zeta$. Hence $a_{-m-1}(w)$ is well-defined on
$Y_{\Gamma}$ as the residue of this differential form at $\Gamma w \in
Y_{\Gamma}$. This proves the corollary.
\end{proof}

\section{The Regularized Pairing of Bringmann--Kane \label{RegPair}}

\cite{[BK]} introduces a regularization which gives meaning to integrals pairing
modular forms with singularities in $\mathcal{H}$. This regularization makes use
of the coordinate $A_{w}(z)$ around singular points $w$. Explicitly, let two
modular forms $f$ and $g$ of the same weight $k$ with respect to a Fuchsian
group $\Gamma$ of the first kind be given. We allow both $f$ and $g$ to have
(isolated) singularities in $\mathcal{H}$. Fix a (nice enough) fundamental
domain $\mathcal{F}$ for $\Gamma$, and let $w_{j}$, $1 \leq j \leq l$ be the
singular points of $f$ and of $g$ in $\mathcal{F}$. We assume that each $w_{j}$
is an inner point of the union of the images of $\mathcal{F}$ under the
stabilizer $\Gamma_{w_{j}}$ of $w_{j}$ in $\Gamma$, and that
$f(z)\overline{g(z)}\big|A_{w_{j}}(z)\big|^{s_{j}}v^{k}$ is a continuous
function of $z$ in a neighborhood of $w_{j}$ for $s_{j}$ in some right
half-plane in $\mathbb{C}$. One then defines the (regularized) pairing $\langle
f,g \rangle^{reg}$ of $f$ and $g$ by considering the integral
\begin{equation}
\int_{\mathcal{F}}f(z)\overline{g(z)}\prod_{j=1}^{l}\big|A_{w_{j}}(z)\big|^{s_
{j}}v^{k}d\mu(z) \label{regdef}
\end{equation}
(where $d\mu(z)=\frac{dudv}{v^{2}}$ is the invariant measure on $\mathcal{H}$),
extending it to a meromorphic function of
$\mathbf{s}=\{s_{j}\}_{j=1}^{l}\in\mathbb{C}^{l}$ (if such a meromorphic
continuation exists), and taking the constant term of the resulting Laurent
expansion at $\mathbf{s}=0$. At this point we assume that $f\overline{g}$ decreases rapidly enough at the cusps of $\Gamma$ (if they exist), so that there are no convergence problems arising from them. The more general case, involving wilder behavior at the cusps, will be considered below. This pairing is well-defined by the following
\begin{prop}
The pairing of $f$ and $g$ is independent of the choice of the fundamental
domain $\mathcal{F}$. \label{pairwd}
\end{prop}

\begin{proof}
If we change $\mathcal{F}$ in a manner which leaves all the singular points
invariant (i.e., does not take any of the them to a non-trivial image under
$\Gamma$) then this follows as for the independence of the Petersson inner
product of the fundamental domain (since we take the value at
$\mathbf{s}=0$). Corollary \ref{absAwzinv} allows us to move the singular points
as well, which completes the proof of the proposition.
\end{proof}

\smallskip

We now present a tool which will make our evaluation of this pairing much
simpler. For any $w\in\mathcal{H}$ and $\varepsilon>0$ define
\[\mathcal{D}_{\varepsilon,w}=\big\{z\in\mathcal{H}\big||A_{w}
(z)|<\varepsilon\big\}=A_{w}^{-1}(\mathcal{B}_{\varepsilon}).\] The fact that
$\mathcal{F}$ contains only finitely many singular points implies that for small
enough $\varepsilon$ the sets $\mathcal{D}_{\varepsilon,w_{j}}$, $1 \leq j \leq
l$ are pairwise disjoint. Note that our assumption on the relation between the
$w_{j}$s and $\mathcal{F}$ implies that
\begin{equation}
\mathcal{D}_{\varepsilon,w_{j}}=\bigcup_{\gamma\in\Gamma_{w_{j}}}
\gamma\big(\mathcal{D}_{\varepsilon,w_{j}}\cap\mathcal{F}\big) \label{Dwjunion}
\end{equation}
if $\varepsilon$ is small enough, so that the image of
$\mathcal{F}\cap\mathcal{D}_{\varepsilon,w_{j}}$ in $Y_{\Gamma}$ is a full
neighborhood (which we denote by $\mathcal{D}_{\varepsilon,\Gamma w_{j}}$) of the
image $\Gamma w_{j}$ of $w_{j}$ in $Y_{\Gamma}$ (as well as in its compactification $X_{\Gamma}$ obtained by adding the cusps to $Y_{\Gamma}$). It is clear that the set $\mathcal{F}_{\varepsilon}=\mathcal{F}\setminus\bigcup_{j=1}^{l}\mathcal{D}_{
\varepsilon,w_{j}}$ projects, up to the finitely many cusps, onto the complement $X_{\Gamma,\varepsilon}$ of $\bigcup_{j=1}^{l}\mathcal{D}_{\varepsilon,\Gamma w_{j}}$ in $X_{\Gamma}$. We now establish the following
\begin{lem}
The pairing $\langle f,g \rangle^{reg}$ decomposes as
\[\int_{X_{\Gamma,\varepsilon}}f(z)\overline{g(z)}v^{k}d\mu(z)+\sum_{j=1}^{l}
\mathrm{CT}_{s=0}\int_{\mathcal{D}_{\varepsilon,\Gamma
w_{j}}}f(z)\overline{g(z)}\big|A_{w_{j}}(z)\big|^{s}v^{k}d\mu(z),\] where
$\mathrm{CT}_{s=0}$ means the constant term of the meromorphic continuation in
$s$ at $s=0$. \label{singsep}
\end{lem}

\begin{proof}
We decompose the integral over $\mathcal{F}$ appearing in Equation
\eqref{regdef} as the sum of the integral over $\mathcal{F}_{\varepsilon}$ and
the integrals over $\mathcal{F}\cap\mathcal{D}_{\varepsilon,w_{j}}$. Now,
the integral over $\mathcal{F}_{\varepsilon}$ yields an entire function of
$\mathbf{s}\in\mathbb{C}^{l}$ (in which substituting $\mathbf{s}=0$ yields the
first term), and the integral over
$\mathcal{F}\cap\mathcal{D}_{\varepsilon,w_{j}}$ is entire, for every value of
$s_{j}$, in the other coordinates of $\mathbf{s}$. We make the index change
$s=s_{j}$, and identify the integral over $\mathcal{F}_{\varepsilon}$ with the one over $X_{\Gamma,\varepsilon}$ (trivial). Moreover, the argument which used
Corollary \ref{absAwzinv} to prove Proposition \ref{pairwd} shows that the
integral over $\mathcal{F}\cap\mathcal{D}_{\varepsilon,w_{j}}$ coincides with
that over $\mathcal{D}_{\varepsilon,\Gamma w_{j}}$. This proves the lemma.
\end{proof}

\smallskip

We will be interested in the case where $k=2m+2$, $\Gamma=\Gamma_{L}$ of an even
lattice $L$ of signature $(2,1)$, $g$ is meromorphic, and $f$ is the function
$\frac{1}{2\pi i}\delta_{2m}\Phi_{m,r,\beta}^{L}$ from Theorem \ref{delta2mPhi}.
Let $w_{j}=\sigma_{j}+it_{j}$, $1 \leq j \leq l$ be the poles of $\frac{1}{2\pi
i}\delta_{2m}\Phi_{m,r,\beta}^{L}$ and of $g$ which lie in $\mathcal{F}$. We now
simplify the expression for the pairing in question as in the following
\begin{prop}
The pairing of $\big\langle\frac{1}{2\pi
i}\delta_{2m}\Phi_{m,r,\beta}^{L},g\big\rangle^{reg}$ can be written as
\[-\sum_{j=1}^{l}\mathrm{CT}_{s=0}\frac{s}{8\pi}\int_{\mathcal{D}_{\varepsilon,
\Gamma_{L}w_{j}}}\Phi_{m,r,\beta}^{L}(z)\overline{g(z)}\big|A_{w_{j}}(z)\big|^{s
}\frac{2it_{j}}{(z-w_{j})(z-\overline{w_{j}})}v^{2m}dzd\overline{z}.\]
\label{afterStokes}
\end{prop}

\begin{proof}
We write $v^{2m+2}d\mu(z)$ as $\frac{i}{2}v^{2m}dzd\overline{z}$, and observe
that the non-regularized part $\frac{1}{2\pi
i}\delta_{2m}\Phi_{m,r,\beta}^{L}(z)\overline{g(z)}v^{2m}dzd\overline{z}$ of the
$(1,1)$-form that we are integrating in Lemma \ref{singsep} is exact. Indeed,
$v^{2m}\delta_{2m}\Phi_{m,r,\beta}^{L}$ is
$\partial_{z}\big(v^{2m}\Phi_{m,r,\beta}^{L}\big)$, and we can put
$\overline{g(z)}$ inside the derivative since $g$ is meromorphic. Pulling in
$d\overline{z}$, this differential form is
$\frac{1}{4\pi}d\big(\Phi_{m,r,\beta}^{L}(z)\overline{g(z)}v^
{2m}d\overline{z}\big)$. Hence this is the integrand in the first term in
Lemma \ref{singsep}, while the integrand appearing in the $j$th summand in the
second term is the same expression but multiplied by
$\big|A_{w_{j}}(z)\big|^{s}$. We decompose the latter product according to the
rule
\begin{equation}
d\big(H(z)d\overline{z}\big)\big|A_{w_{j}}(z)\big|^{s}=d\big(H(z)\big|A_{w_{j}
}(z)\big|^{s}d\overline{z}\big)-H(z)\partial_{z}\big|A_{w_{j}}(z)\big|^{s}
dzd\overline{z}, \label{ddiff}
\end{equation}
and now apply Stokes' Theorem for the integrals involving
exact differential forms. The first term from Lemma \ref{singsep} thus yields
\[\frac{1}{4\pi}\int_{\partial
X_{\Gamma,\varepsilon}}\Phi_{m,r,\beta}^{L}(z)\overline{g(z)}v^
{2m}d\overline{z},\] while Equation \eqref{ddiff} shows that the integral in the
$j$th summand becomes $\frac{1}{4\pi}$ times
\[\int_{\partial\mathcal{D}_{\varepsilon,\Gamma_{L}w_{j}}}\!\!\!\!\Phi_{m,r,
\beta}^{L}(z)\overline{g(z)}\big|A_{w_{j}}(z)\big|^{s}\!v^{2m}\!d\overline{z}
-\!\int_{\mathcal{D}_{\varepsilon,\Gamma_{L}w_{j}}}\!\!\!\!\Phi_{m,r,\beta}^{L}
(z)\overline{g(z)}\partial_{z}\big|A_{w_{j}}(z)\big|^{s}\!v^{2m}\!dzd\overline{z
}.\] For the constant term at $s=0$ we may just substitute $s=0$ at the integral
over $\partial\mathcal{D}_{\varepsilon,\Gamma_{L}w_{j}}$. This yields the same
integrand as in the integral over $\partial X_{\Gamma,\varepsilon}$, where the
latter boundary is the sum of the former boundaries but with the opposite
orientation. Hence all these terms cancel. Applying part $(ii)$ of Lemma
\ref{Awprop} for the derivative of
$\big|A_{w_{j}}(z)\big|^{s}=\big(A_{w_{j}}(z)\overline{A_{w_{j}}(z)}\big)^{\frac{s}{2}}$
now yields
\[\frac{s}{2}\big|A_{w_{j}}(z)\big|^{s-2}\overline{A_{w_{j}}(z)}\frac{2it_{j}}{
(z-\overline{w_{j}})^{2}},\quad\mathrm{which\
equals}\quad\frac{s}{2}\big|A_{w_{j}}(z)\big|^{s}\frac{2it_{j}}{(z-w_{j}
)(z-\overline{w_{j}})}\] by the definition of $A_{w_{j}}(z)$. This completes the
proof of the proposition.
\end{proof}

\section{Unfolding Again \label{UnfPhi}}

The sets $S_{\beta,r}$ and $S_{-\beta,r}$ appearing in the expression for
$\Phi_{m,r,\beta}^{L}(z)$ given in Corollary \ref{b-=1lift} consist of finitely
many orbits of $\Gamma_{L}$. For simplicity of the following presentation, we
shall assume, for the moment, that $S_{-\beta,r}$ is empty and $S_{\beta,r}$
consists of a single orbit $\Gamma_{L}w_{0}$ of $\Gamma_{L}$ (the general case
will be obtained via a simple summation). It thus makes sense to choose a
representative ($w_{0}$, say), and unfold the integral from Proposition
\ref{afterStokes}. This is also done in \cite{[BK]}, where expressions based on
the function $f_{D,k,[Q]}$ are given in terms of a representing quadratic form
$Q$ of the class $[Q]$. The result here becomes
\begin{prop}
The pairing $\big\langle g,\frac{1}{2\pi i}\delta_{2m}\Phi_{m,r,\beta}^{L}
\big\rangle^{reg}$ equals the sum over all the points
$\tilde{w}=\tilde{\sigma}+i\tilde{t}\in\Gamma_{L}\{w_{j}\}_{j=0}^{l}$ of the
constant term at $s=0$ of $\frac{-2|r|^{m/2}m!s}{(-8\pi
i)^{m+1}|\Gamma_{L,w_{0}}|}$ times
\[\int_{\mathcal{D}_{\varepsilon,\tilde{w}}}\frac{(z-w_{0})^{m}(z-\overline{w_{0
}})^{m}\tilde{t}}{t_{0}^{m}(\overline{z}-\tilde{w})(\overline{z}-\overline{
\tilde{w}})}B_{m}\big(\cosh
d(z,w_{0})\big)g(z)\big|A_{\tilde{w}}(z)\big|^{s}dzd\overline{z}.\]
\label{unfw0}
\end{prop}

\begin{proof}
We plug the formula from Corollary \ref{b-=1lift} into the expression from
Proposition \ref{afterStokes}, and take the complex conjugate since we have
interchanged $g$ and $\frac{1}{2\pi i}\delta_{2m}\Phi_{m,r,\beta}^{L}$ in the
pairing (note that this leaves the measure $idzd\overline{z}=2dudv$ invariant).
After replacing the integration domain by
$\mathcal{F}\cap\mathcal{D}_{\varepsilon,w_{j}}$ and replacing the sum
over the orbit by the sum over $\Gamma_{L}$, we get for each $0 \leq j \leq l$
the coefficient $\frac{-2|r|^{m/2}m!s}{(-8\pi i)^{m+1}|\Gamma_{L,w_{0}}|}$ times
\[\sum_{\gamma\in\Gamma_{L}}\int_{\mathcal{F}\cap\mathcal{D}_{\varepsilon,w_{j}}
}\frac{(z-w)^{m}(z-\overline{w})^{m}t_{j}}{t^{m}(\overline{z}-w_{j})(\overline{z
}-\overline{w_{j}})}B_{m}\big(\cosh
d(z,w)\big)g(z)\big|A_{w_{j}}(z)\big|^{s}dzd\overline{z},\] where
$w=\gamma^{-1}w_{0}$ (hence we divided by the size of $\Gamma_{L,w_{0}}$). Now,
Proposition 3.3 of \cite{[Ze3]} and Lemma \ref{b-=1neg} allow us to replace
$\frac{(z-w)^{m}(z-\overline{w})^{m}}{t^{m}}$ and
$\frac{t_{j}}{(\overline{z}-w_{j})(\overline{z}-\overline{w_{j}})}$ by
\[\frac{(\gamma z-w_{0})^{m}(\gamma
z-\overline{w_{0}})^{m}}{t_{0}^{m}}j(\gamma,z)^{2m}\quad\mathrm{and}\quad\frac{
\Im\gamma w_{j}}{(\gamma\overline{z}-\gamma
w_{j})(\gamma\overline{z}-\gamma\overline{w_{j}})\overline{j(\gamma,z)^{2}}}\]
respectively, and we may replace $g(z)$ by $\frac{g(\gamma
z)}{j(\gamma,z)^{2m+2}}$ by the modularity of $g$. We also apply Corollary
\ref{absAwzinv} for $|A_{w_{j}}(z)|$ and use the invariance of the hyperbolic
distance. The integral in the summand corresponding to $\gamma$ then takes (after all the cancelations) the form
\[\int_{\mathcal{F}\cap\mathcal{D}_{\varepsilon,w_{j}}}\!\!\!\frac{(\gamma
z\!-\!w_{0})^{m}(\gamma z\!-\!\overline{w_{0}})^{m}\Im\gamma
w_{j}}{t_{0}^{m}(\gamma\overline{z}-\gamma
w_{j})(\gamma\overline{z}-\gamma\overline{w_{j}})}B_{m}\big(\cosh d(\gamma
z,w_{0})\big)g(\gamma z)\frac{\big|A_{\gamma w_{j}}(\gamma
z)\big|^{s}dzd\overline{z}}{|j(\gamma,z)|^{4}}.\] We apply the usual change of
variable to get an integral over
$\gamma\big(\mathcal{F}\cap\mathcal{D}_{\varepsilon,w_{j}}\big)$, and using
Equation \eqref{Dwjunion} we find that the total domain of integration arising
from all $\gamma\in\Gamma_{L}$ for which $\gamma w_{j}=\tilde{w}$ for some
$\tilde{w}\in\Gamma_{L}w_{j}$ is precisely $D_{\varepsilon,\tilde{w}}$. Summing
over $\gamma\in\Gamma_{L}$ and $0 \leq j \leq l$ now completes the proof of the
proposition.
\end{proof}
Recall that we consider $\Gamma_{L}$ as a subgroup of $SO^{+}(L_{\mathbb{R}})$,
i.e., of $PSL_{2}(\mathbb{R})$. Hence the size of a generic stabilizer is 1,
rather than 2 as in subgroups of $SL_{2}(\mathbb{R})$.

\smallskip

Proposition \ref{unfw0} presents the pairing as the sum of the contributions
from (neighborhoods around) the poles of $g$ and of $\frac{1}{2\pi
i}\delta_{2m}\Phi_{m,r,\beta}^{L}$. We shall evaluate the two contributions
separately, as they present a slightly different behavior. In fact, the
presentation given in Proposition \ref{unfw0} reduces the examination of the
poles of the latter function to a single one $w_{0}$. For analyzing it we shall need an explicit formula for our function $B_{m}(T)$ (this is also useful when one carries out the comparison with higher Green's functions mentioned below). This is given in
\begin{lem}
The function $B_{m}(T)$ can be written explicitly as
\[\sum_{h=0}^{m-1}(-1)^{h}\frac{(2m)!(m-1-h)!T}{4^{h}(2m-2h)!m!(T^{2}-1)^{m-h}}
+(-1)^{m}\frac{(2m)!}{4^{m}m!}\ln\bigg(\frac{T+1}{T-1}\bigg).\] It thus extends to a holomorphic function of $T\in\mathbb{C}$ with $\Re T>1$. \label{singB}
\end{lem}

\begin{proof}
Examining the derivative of the function $\frac{T}{(T^{2}-1)^{m}}$ gives us the
equality
\[B_{m}(T)=\frac{T}{m(T^{2}-1)^{m}}-\frac{2m-1}{2m}B_{m-1}(T)\] (this can also
be easily seen if one applies integration by parts to the integral defining
$B_{m}(T)$ as an incomplete beta function). Apply this equality $m$ times, and
use the fact that
\[B_{0}(T)=\int_{T}^{\infty}\frac{2d\xi}{\xi^{2}-1}=\ln\bigg(\frac{T+1}{
T-1}\bigg).\] The extension to a holomorphic function is possible either using the expression given here or already from the integral defining $B_{m}$ in part $(i)$ of Lemma \ref{Bm1/T2}. This completes the proof of the lemma.
\end{proof}

The contribution from the pole at $w_{0}$ is now given in
\begin{prop}
The summand arising from $\tilde{w}=w_{0}$ in Proposition \ref{unfw0} gives a
holomorphic function of $s\in\mathbb{C}$ with $\Re s>2m$, whose meromorphic
continuation has a vanishing constant term at $s=0$. \label{contw0}
\end{prop}

\begin{proof}
We expand $g(z)$ as in Equation \eqref{expwzeta} with $\zeta=A_{w_{0}}(z)$, and
apply part $(iv)$ of Lemma \ref{Awprop} to the expressions appearing in the
integral from Proposition \ref{unfw0} (including $\cosh d(z,w_{0})$ from
Equation \eqref{coshd}). After all the cancelations, the integral becomes
\[-(2i)^{m-1}\int_{\mathcal{B}_{\varepsilon}}B_{m}\bigg(\frac{1+|\zeta|^{2}}{
1-|\zeta|^{2}}\bigg)\sum_{n\gg-\infty}a_{n}(w_{0})\zeta^{n+m+1}|\zeta|^{s-2}
d\zeta d\overline{\zeta}.\] We write $\zeta=\rho e^{i\varphi}$, hence $d\zeta
d\overline{\zeta}=-2i\rho d\rho d\varphi$, so that
the latter expression equals
\[(2i)^{m}\int_{0}^{\varepsilon}\int_{0}^{2\pi}B_{m}\bigg(\frac{1+\rho^{2}}{
1-\rho^{2}}\bigg)\sum_{n\gg-\infty}a_{n}(w_{0})e^{i(n+m+1)\varphi}\rho^{n+m+s}
d\rho d\varphi.\] Plugging in the coefficient
$\frac{-2|r|^{m/2}m!s}{(-8\pi i)^{m+1}|\Gamma_{L,w_{0}}|}$ appearing in
Proposition \ref{unfw0} and carrying out the integration over $\varphi$ (which
leaves only the term with $n=-m-1$) reduces us to the expression
\[\frac{-i|r|^{m/2}m!s}{2(-4\pi)^{m}|\Gamma_{w_{0}}|}\int_{0}^{\varepsilon}B_{m}
\bigg(\frac{1+\rho^{2}}{1-\rho^{2}}\bigg)a_{-m-1}(w_{0})\rho^{s-1}d\rho\] (note
the independence of the choice of the representative $w_{0}$ of the orbit, by
Corollary \ref{a-m-1} and conjugation for the size of the stabilizer). Since
$B_{m}$ is bounded on closed intervals of $[0,1]$ not containing 1, integrating
$B_{m}\big(\frac{1+\rho^{2}}{1-\rho^{2}}\big)\rho^{s-1}$ between $\varepsilon$
and 1 yields an entire function of $s\in\mathbb{C}$. As we multiply our integral
by $s$, the constant term at $s=0$ in question does not change if we replace
$\varepsilon$ by 1. We now substitute $T=\frac{1+\rho^{2}}{1-\rho^{2}}$ in Lemma
\ref{singB}. Each quotient of the form $\frac{T}{(T^{2}-1)^{m-h}}$ takes the
form $\frac{(1+\rho^{2})(1-\rho^{2})^{2m-2h-1}}{(2\rho)^{2m-2h}}$, and the
argument of the logarithm is just $\frac{1}{\rho^{2}}$. All these expressions
vanish at $\rho=1$, and their product with $\rho^{s}$ for $s\in\mathbb{C}$ with
$\Re s>2m$ tend to 0 as $\rho\to0^{+}$. For such $s$ we may apply integration by
parts and use Lemma \ref{Bm1/T2} to get
\[\int_{0}^{1}B_{m}\bigg(\frac{1+\rho^{2}}{1-\rho^{2}}\bigg)s\rho^{s-1}d\rho=
\int_{0}^{1}\frac{2}{(\xi^{2}-1)^{m+1}}\bigg|_{\xi=\frac{1+\rho^{2}}{1-\rho^{2}}
}\cdot\frac{4\rho}{(1-\rho^{2})^{2}}\rho^{s}d\rho,\] since
$B_{m}\big(\frac{1+\rho^{2}}{1-\rho^{2}}\big)\rho^{s}$ was seen to vanish at the
two limits of the integral. Substituting, and using the Binomial Theorem, the
integrand becomes
\[2^{1-2m}(1-\rho^{2})^{2m}\rho^{s-1-2m}=2^{1-2m}\sum_{l=0}^{2m}\binom{2m}{l}
(-1)^{l}\rho^{s-1-2m+2l}.\] Integrating (which we can do for $\Re s>2m$), we
find that for any $0<l<m$ the terms arising from $l$ and $2m-l$ yield the
functions $\frac{1}{s-2m+2l}$ and $\frac{1}{s+2m-2l}$, both multiplied by the
same coefficient $(-1)^{l}\binom{2m}{l}$. As these functions are holomorphic at
$s=0$, with constant terms which are additive inverses, the contribution of each
such pair to the constant term at $s=0$ cancels. The remaining term, with $l=m$,
is just a multiple of $\frac{1}{s}$, whose constant term at $s=0$ vanishes. This
completes the proof of the proposition.
\end{proof}

\section{Poles at the Cusps \label{Cuspz}}

In this section we assume that $\Gamma$ has cusps. Then a meromorphic modular
form $f$ of weight $k$ with respect to $\Gamma$ might have poles at the cusps,
so that regularizing the integral there may also be required. For each cusp
$\kappa$, we choose a matrix $A_{\kappa} \in PSL_{2}(\mathbb{R})$ with
$A_{\kappa}\kappa=\infty$. Then $f[A_{\kappa}^{-1}]_{k}$ is $h$-periodic for
some positive number $h$, hence admits a Fourier expansion in $\mathbf{e}(z/h)$.
The sesqui-linear product of two such modular forms (times $v^{k}$) grows
exponentially at the cusp, but following Section 9 of \cite{[BK]} we observe
that multiplying by $e^{-sv}$ gives a bounded function as $y\to\infty$, if $\Re
s$ is large enough. If $\big\{\kappa_{j}\big\}_{j=1}^{\tilde{l}}$ are the cusps
in the fundamental domain $\mathcal{F}$ we chose for $\Gamma$, then we multiply
the integrand from Equation \eqref{regdef} also by
$\prod_{j=1}^{\tilde{l}}e^{-\tilde{s}_{j}\Im A_{\kappa_{j}}z}$. The integral
thus gives a holomorphic function also of the $\tilde{s}_{j}$s in some right
half-plane, and we extend the definition of the regularized pairing to this case
by taking the constant term at the point where all the $\tilde{s}_{j}$s also
vanish. We now have
\begin{prop}
The regularized integral is independent of the choice of the matrices
$A_{\kappa_{j}}$, as well as of the fundamental domain. \label{regcusp}
\end{prop}

\begin{proof}
The only possible change to $A_{\kappa_{j}}$ is to multiply it from the left by
a matrix of the form $\binom{a\ \ \ b\ \ }{0\ \ a^{-1}}$ for some $a>0$ and
$b\in\mathbb{R}$. This replaces $A_{\kappa_{j}}z$ by $a^{2}A_{\kappa_{j}}z+ab$,
hence multiplies $\Im A_{\kappa_{j}}z$ by $a^{2}$. The resulting function of
$\tilde{s}_{j}$ is hence the same function, but evaluated at
$a^{2}\tilde{s}_{j}$. As the constant term at $\tilde{s}_{j}=0$ remains
invariant under this operation, this proves the first assertion. Proposition
\ref{pairwd} shows the invariance of the pairing under replacing the fundamental
domain by another fundamental domain having the same cusps. Now, if
$\kappa=\gamma\lambda$ with $\lambda$ being another cusp and $\gamma\in\Gamma$
then the matrix $A_{\kappa}\gamma$ may be used as $A_{\lambda}$. Combining this
fact with the argument proving the Proposition \ref{pairwd} establishes the
desired invariance also in the case where we do move the cusps in the choice of
the fundamental domain. This proves the proposition.
\end{proof}

\smallskip

We assume that for any cusp $\kappa$ of the fundamental domain $\mathcal{F}$,
the union of the translates of $\mathcal{F}$ by the elements of the stabilizer
$\Gamma_{\kappa}$ of $\kappa$ in $\Gamma$ contains the inverse image under
$A_{k}$ of a set of the form $\big\{z\in\mathcal{H}\big|v>M\big\}$ for some
(large) $M>0$. Given $\varepsilon>0$ and a choice of a matrix $A_{\kappa}$ for
some cusp $\kappa$, we define
\[\mathcal{D}_{\varepsilon,\kappa}=\big\{z\in\mathcal{H}\big||\mathbf{e}(A_{
\kappa}z)|<\varepsilon\big\}=\bigg\{z\in\mathcal{H}\bigg|\Im
A_{\kappa}z>\frac{\ln(1/\varepsilon)}{2\pi}\bigg\}.\] For small enough
$\varepsilon$, the equivalent of Equation \eqref{Dwjunion} holds for cusps, and
$\mathcal{D}_{\varepsilon,\kappa}\cap\mathcal{F}$ maps onto a full punctured
neighborhood $\mathcal{D}_{\varepsilon,\Gamma\kappa}$ of the cusp
$\Gamma\kappa$ of $X_{\Gamma}$. In addition, if $\varepsilon$ is small enough
then the neighborhoods $\mathcal{D}_{\varepsilon,\Gamma\kappa_{j}}$ are all
disjoint and do not intersect the neighborhoods
$\mathcal{D}_{\varepsilon,\Gamma w_{j}}$ of the poles of $f$ and $g$. Extending
the definition of $\mathcal{F}_{\varepsilon}$ and $X_{\Gamma,\varepsilon}$ to
this case (with the neighborhoods around the cusps also removed), the
expression from Lemma \ref{singsep} remains valid also here if we add
\[\sum_{j=1}^{\tilde{l}}\mathrm{CT}_{s=0}\int_{\mathcal{D}_{\varepsilon,
\Gamma\kappa_{j}}}f(z)\overline{g(z)}e^{-s\Im A_{\kappa_{j}}z}v^{k}d\mu(z)\] to
it. In case $k=2m+2$, $\Gamma=\Gamma_{L}$, and $f=\frac{1}{2\pi
i}\delta_{2m}\Phi_{m,r,\beta}^{L}$ we find
\begin{prop}
If $\Gamma_{L}$ has cusps then the pairing $\big\langle g,\frac{1}{2\pi
i}\delta_{2m}\Phi_{m,r,\beta}^{L} \big\rangle^{reg}$ is given by the expression
from Proposition \ref{afterStokes} plus
\[\sum_{j=1}^{\tilde{l}}\mathrm{CT}_{s=0}\frac{s}{8\pi}\int_{\mathcal{D}_{
\varepsilon,\Gamma_{L}\kappa_{j}}}\Phi_{m,r,\beta}^{L}(z)\overline{g(z)}\frac{e^
{-s\Im A_{\kappa_{j}}z}}{j(A_{\kappa_{j}},z)^{2}}v^{2m}idzd\overline{z}.\]
\label{Stokescusp}
\end{prop}

\begin{proof}
We use the same argument from the proof of Proposition \ref{afterStokes}.
Note that $\partial X_{\Gamma_{L},\varepsilon}$ contains the boundaries of both
the neighborhoods $\mathcal{D}_{\varepsilon,\Gamma_{L}w_{j}}$ of the poles and
the neighborhoods $\mathcal{D}_{\varepsilon,\Gamma_{L}\kappa_{j}}$ of the cusps.
We thus apply Equation \eqref{ddiff} also for the integral over
$\mathcal{D}_{\varepsilon,\Gamma_{L}\kappa_{j}}$, and after applying Stokes'
Theorem, all the integrals over the boundaries vanish. The remaining integrals
over the neighborhoods $\mathcal{D}_{\varepsilon,\Gamma_{L}w_{j}}$ are evaluated
as in Proposition \ref{afterStokes}, while for the integral over
$\mathcal{D}_{\varepsilon,\Gamma_{L}\kappa_{j}}$ we evaluate
$\partial_{z}e^{-s\Im A_{\kappa}z}$ as \[-se^{-s\Im
A_{\kappa}z}\partial_{z}\frac{v}{|j(A_{\kappa},z)|^{2}}=\frac{-se^{-s\Im A_{\kappa}z}\big(j(A_{\kappa},z)-2ivj'(A_{\kappa},z)\big)}{2i|j(A_{\kappa},z)|^{2}j(A_{\kappa},z)}=\frac{-se^{-s\Im A_{\kappa}z}}{2ij(A_{\kappa},z)^{2}},\] where $j'(A_{\kappa},z)$ is just a scalar (the $c$-entry of $A_{\kappa}$). Recalling the external coefficient $\frac{1}{4\pi}$, this completes the proof of the proposition.
\end{proof}

\smallskip

The unfolding process which we carry out for the cusps is a bit different. For
any $1 \leq j\leq\tilde{l}$ we define $S_{j}$ to be the $A_{\kappa_{j}}$-image
of a set of representatives for $S_{\beta,r}\cup(-1)^{m}S_{-\beta,r}$ modulo the
action of the infinite cyclic group $\Gamma_{L,\kappa_{j}}$. We then prove
\begin{prop}
If $\Gamma_{L}$ has cusps then the value of $\big\langle g,\frac{1}{2\pi
i}\delta_{2m}\Phi_{m,r,\beta}^{L} \big\rangle^{reg}$ is obtained by adding
the constant term at $s=0$ of $\frac{-4i|r|^{m/2}m!s}{(-8\pi
i)^{m+1}}$ times \[\sum_{j=1}^{\tilde{l}}\sum_{w \in
S_{j}}\int_{M}^{\infty}\int_{-\infty}^{\infty}\frac{(z-w)^{m}(z-\overline{w})^{m
}}{t^{m}}B_{m}\big(\cosh
d(z,w)\big)g[A_{\kappa_{j}}^{-1}]_{2m+2}(z)e^{-sv}dudv\] to the expression from
Proposition \ref{unfw0}. \label{unfcusp}
\end{prop}

Note that multiplying $A_{\kappa_{j}}$ by $\binom{a\ \ \ b\ \ }{0\ \ a^{-1}}$
from the left just replaces the variable $s$ by $a^{2}s$ (hence leaves the
constant term in question invariant), as one easily sees by a simple change of
variables.

\begin{proof}
The same argument as in the proof of Proposition \ref{unfw0} (but with leaving
the summation on $w$ rather than on $\gamma$) shows that the $j$th term from
Proposition \ref{Stokescusp} can be written as $\frac{2|r|^{m/2}m!s}{(-8\pi
i)^{m+1}}$ times
\[\sum_{w}\int_{\mathcal{F}\cap\mathcal{D}_{\varepsilon,\kappa_{j}}}\frac{(z-w)^
{m}(z-\overline{w})^{m}}{t^{m}\overline{j(A_{\kappa_{j}},z)}^{2}}B_{m}\big(\cosh
d(z,w)\big)g(z)e^{-s\Im A_{\kappa_{j}}z}dzd\overline{z},\] where $w$ is taken
from $S_{\beta,r}\cup(-1)^{m}S_{-\beta,r}$ as above. As in the proof of
Proposition \ref{unfw0}, we apply again a change of variable, but this time with
respect to $A_{\kappa_{j}}$. We write $g(z)$ as
$\frac{g[A_{\kappa_{j}}^{-1}]_{2m+2}(A_{\kappa_{j}}z)}{j(A_{\kappa_{j}},z)^{2m+2
}}$ and
\[\frac{(z-w)^{m}(z-\overline{w})^{m}}{t^{m}}=\frac{(A_{\kappa_{j}}z-A_{\kappa_{
j}}w)^{m}(A_{\kappa_{j}}z-A_{\kappa_{j}}\gamma\overline{w})^{m}}{(\Im
A_{\kappa_{j}}w)^{m}}j(A_{\kappa_{j}},z)^{2m},\] and using the invariance of the
hyperbolic distance and the formula for the derivatives in order to write the
latter sum as
\[\sum_{\tilde{w}}\int_{A_{\kappa_{j}}(\mathcal{F}\cap\mathcal{D}_{\varepsilon,
\kappa_{j}})}\frac{(z-\tilde{w})^{m}(z-\overline{\tilde{w}})^{m}}{\tilde{t}^{m}
}B_{m}\big(\cosh
d(z,\tilde{w})\big)g[A_{\kappa_{j}}^{-1}]_{2m+2}(z)e^{-sv}dzd\overline{z}.\]
Here $\tilde{w}=\tilde{\sigma}+i\tilde{t}=A_{\kappa_{j}}w$ runs over the set
$A_{\kappa_{j}}\big(S_{\beta,r}\cup(-1)^{m}S_{-\beta,r}\big)$. Now,
$A_{\kappa}(\mathcal{F}\cap\mathcal{D}_{\varepsilon,\kappa_{j}})$ is a strip of
width $h$ in $\big\{z\in\mathcal{H}\big|v>M\big\}$ for
$M=\frac{\ln(1/\varepsilon)}{2\pi}$, and the set of points $\tilde{w}$ consists
of orbits of the group $A_{\kappa}\Gamma_{L}A_{\kappa}^{-1}$. The latter group
contains, in particular, the $A_{\kappa}$-conjugate $T^{h}=\binom{1\ \ h}{0\ \
1}$ of the generator of $\Gamma_{L,\kappa}$. We thus sum only over
representatives for the action of $\Gamma_{L,\kappa}$ (note that $\tilde{t}$ is
independent of the choice of the representative), and using the powers of
$T^{h}$ we integrate over the full half-plane of $z\in\mathcal{H}$ with $v>M$.
We now write $w$ instead of $\tilde{w}$
and put the external coefficient back again. This completes the proof of the
proposition.
\end{proof}

\section{Contributions from the Cusps \label{Cuspeval}}

We wish to evaluate the contribution of each summand in Proposition
\ref{unfcusp} explicitly. In order to do this, we shall need the following
formulae:
\begin{lem}
$(i)$ Let a polynomial $Q$, a non-negative real number $\eta$, four distinct
complex, non-real numbers $\kappa$, $\lambda$, $\mu$, and $\nu$, and four
non-negative integers $a$, $b$, $c$, and $d$ be given. Assuming that the degree
of $Q$ does not exceed $a+b+c+d+2$, the integral
\[\int_{-\infty}^{\infty}\frac{Q(u)e^{-i\eta
u}du}{(u-\kappa)^{a+1}(u-\lambda)^{b+1}(u-\mu)^{c+1}(u-\nu)^{d+1}}\] equals
$-2\pi i$ times the sum of the residues of the integrand at the elements of
$\{\kappa,\lambda,\mu,\nu\}$ whose imaginary part is negative. $(ii)$ In the
notation of part $(i)$ we have that
\[\mathrm{Res}_{u=\kappa}\frac{Q(u)e^{-i\eta
u}du}{(u-\kappa)^{a+1}(u-\lambda)^{b+1}(u-\mu)^{c+1}(u-\nu)^{d+1}}\] equals
\[\sum_{p,q,r,k}\!\!\binom{b+p}{p}\!\binom{c+q}{q}\!\binom{d+r}{r}\frac{Q^{(k)}
(\kappa)(i\eta)^{a-p-q-r-k}e^{-i\eta\kappa}/(a-p-q-r-k)!}{(-1)^{a-k}
k!(\kappa-\lambda)^{b+p+1}(\kappa-\mu)^{c+q+1}(\kappa-\nu)^{d+r+1}}\!.\]
\label{residues}
\end{lem}

\begin{proof}
We take the integral from part $(i)$ on the interval $[-R,R]$ for a very large
$R$, and complete it to an integral over a closed path by adding the integral
over the lower part of a circle of radius $R$ centered at 0. As this closed
path is negatively oriented, the closed integral gives the asserted value
(independently of $R$ for $R$ large enough). By taking the limit $R\to\infty$,
we get the integral in question, so that it remains to show that the integral
over the half-circle tends to 0 when $R\to\infty$. But as $\eta\geq0$ we have
$|e^{-i\eta u}|\leq1$ there, where for large enough $R$ the denominator is at
least $CR^{a+b+c+d+4}$ for some constant $C$. In addition, we have
$|Q(u)|<DR^{a+b+c+d+2}$ with another constant $D$ by our assumption on the
degree of $Q$, and the length of the path is $\pi R$. Hence the absolute value
of the half-circular integral is bounded by $\frac{\pi D}{CR}$, which tends to 0
as $R\to\infty$, as desired. This proves part $(i)$. We now note that the
function whose residue we are looking for in part $(ii)$ is of the form
$\frac{h(u)}{(u-\kappa)^{a+1}}$ where $h$ is holomorphic at $\kappa$. Hence this
residue is the $a$th derivative of $h$, divided by $a!$. Using the Multinomial
Theorem for derivatives (Leibnitz rule for higher order derivatives) and
evaluating the derivatives of $\frac{1}{(u-\lambda)^{b+1}}$,
$\frac{1}{(u-\mu)^{c+1}}$, $\frac{1}{(u-\nu)^{d+1}}$, $e^{-i\eta u}$, and $Q(u)$
yields the desired result.
This proves the proposition.
\end{proof}

\smallskip

The function $g[A_{\kappa_{j}}^{-1}]_{2m+2}$ is $h$-periodic and has at most a
pole at the cusp $A_{\kappa_{j}}\kappa_{j}=\infty$. Its Fourier expansion is
thus of the form $\sum_{n\gg-\infty}a_{n}(\kappa_{j})\mathbf{e}(nz/h)$. Plugging
this expansion into the expression from Proposition \ref{unfcusp} shows that we
have to examine integrals of the form
\[\int_{M}^{\infty}\int_{-\infty}^{\infty}\frac{(z-w)^{m}(z-\overline{w})^{m}}{
t^{m}}B_{m}\big(\cosh d(z,w)\big)e^{2\pi inz/h}du \cdot e^{-sv}dv.\] For these
integrals we shall use
\begin{prop}
For every non-negative integer $n$, the expression
\[\int_{-\infty}^{\infty}\frac{(z-w)^{m}(z-\overline{w})^{m}}{t^{m}}B_{m}
\big(\cosh d(z,w)\big)e^{-2\pi inz/h}du\] can be written, for fixed large $v$,
as $\frac{C_{n}}{v}$ plus some term of growth order
$O\big(\frac{1}{v^{2}}\big)$, where $C_{n}$ is a constant. \label{inteval}
\end{prop}

\begin{proof}
The function $(z-w)^{m}(z-\overline{w})^{m}e^{-2\pi inz/h}$ is the derivative of
a function of the sort $\widetilde{Q}_{n}(z)e^{-2\pi inz/h}$, where
$\widetilde{Q}_{n}$ is a polynomial whose degree is $2m+1$ if $n=0$ and $2m$
otherwise. As this function is holomorphic, its derivative with respect to $u$
coincides with its derivative with respect to $z$. We thus integrate by parts to
get that our integral equals \[\frac{\widetilde{Q}_{n}(z)B_{m}\big(\cosh
d(z,w)\big)}{t^{m}e^{2\pi
inz/h}}\bigg|_{u=-\infty}^{u=\infty}-\int_{-\infty}^{\infty}\frac{\widetilde{Q}_
{n}(z)}{t^{m}}e^{-2\pi inz/h}\partial_{u}B_{m}\big(\cosh d(z,w)\big)du.\] Now,
part $(ii)$ of Lemma \ref{Bm1/T2} and Equation \eqref{coshd} show that
$B_{m}\big(\cosh d(z,w)\big)$ decays as
$O\Big(\frac{(2tv)^{m+1}}{(|z|^{2}-2\sigma
u+|w|^{2})^{m+1}}\Big)=O\big(\frac{1}{u^{2m+2}}\big)$ as $u\to\infty$. As the
degree of $\widetilde{Q}_{n}$ is smaller than $2m+2$ and $|e^{-2\pi
inz/h}|=e^{2\pi nv/h}$ is independent of $u$, the first term in the latter
equation vanishes. Using part $(i)$ of Lemma \ref{Bm1/T2} and Equation
\eqref{coshd} for evaluating the expression involving $\partial_{u}B_{m}$ we
find that the expression which we must evaluate is
\[2^{2m+2}v^{2m+1}t^{m+1}\int_{-\infty}^{\infty}\frac{\widetilde{Q}_{n}(z)e^{
-2\pi inz/h}(2u-2\sigma)}{|z-w|^{2m+2}|z-\overline{w}|^{2m+2}}du.\] We decompose
$2u-2\sigma$ as $z-w+z-\overline{w}$ and write, for fixed $v$,
$Q_{n}(u)=\widetilde{Q}_{n}(u+iv)$. By taking out $e^{2\pi nv/h}$ from the
exponent as well, we then get the constant $2^{2m+2}v^{2m+1}t^{m+1}e^{2\pi
nv/h}$ times the sum of two integrals of the form
\[\int_{-\infty}^{\infty}\frac{Q_{n}(u)e^{-2\pi
inu/h}du}{(u+iv-w)^{m+\varepsilon}(u-iv-\overline{w})^{m+1}(u+iv-\overline{w})^{
m+\delta}(u-iv-w)^{m+1}}du,\] once with $\varepsilon=1$ and $\delta=0$, and once
the other way around. This is an integral of the form appearing in part $(i)$ of
Lemma \ref{residues}, with the relevant points in the lower half plane being
$w-iv$ and $\overline{w}-iv$. Applying part $(ii)$ of that lemma with
$\kappa=w-iv$, $\lambda=\overline{w}-iv$, $\mu=w+iv$, $\nu=\overline{w}+iv$,
$\eta=\frac{2\pi n}{h}$, and the integers $a=m+\varepsilon-1$, $b=m+\delta-1$,
and $c=d=m$, we find that the term corresponding to $p$, $q$, $r$, and $k$ is
some combinatorial coefficient times \[\frac{(-1)^{a-k}Q_{n}^{(k)}(w-iv)(2\pi
in/h)^{a-p-q-r-k}e^{-2\pi
in(w-iv)/h}}{(2it)^{b+p+1}(-2iv)^{m+q+1}\big(-2i(v-t)\big)^{m+r+1}}.\]
Interchanging the roles of $\kappa$ and $\lambda$, of $\mu$ and $\nu$, and of
$a$ and $b$ yields the same expression, but with the derivatives of $Q_{n}$
evaluated at $\overline{w}+iv$, with $2it$ replaced by $-2it$, and with $v-t$
replaced by $v+t$.

We investigate the dependence of the resulting expression, multiplied by the
coefficient $2^{2m+2}v^{2m+1}t^{m+1}e^{2\pi nv/h}$, on $v$. First, the exponent
$e^{2\pi nv/h}$ cancels with $e^{-2\pi in\cdot-iv/h}$ from the residues. Second,
as $Q_{n}(\xi)$ is $\widetilde{Q}_{n}(\xi+iv)$, the numerators involve just the
values of $\widetilde{Q}_{n}$ and its derivatives at $w$ and at $\overline{w}$,
which are are independent of $v$. All the terms in which $q+r>0$ have, when
multiplied by $v^{2m+1}$, growth order of at most $O\big(\frac{1}{v^{2}}\big)$.
Moreover, the terms with $q+r=0$ yield some constant $C_{n}$ (depending on $w$,
but not on $v$) times $\frac{v^{m}}{(v \pm
t)^{m+1}}=\frac{1}{v}+O\big(\frac{1}{v^{2}}\big)$. Combining this information
completes the proof of the proposition.
\end{proof}

\smallskip

We can now prove the main result concerning cusps. It is given in
\begin{thm}
The pairing $\big\langle g,\frac{1}{2\pi
i}\delta_{2m}\Phi_{m,r,\beta}^{L}\big\rangle^{reg}$ does not get any
contribution from the regularized integrals at the cusps. \label{cusp0}
\end{thm}

\begin{proof}
We have to prove that the expression from Proposition \ref{unfcusp} vanishes. It
suffices to show that each summand vanishes. Fixing a cusp $\kappa$ and an
element $w\in\mathcal{H}$, we expand $g[A_{\kappa}^{-1}]_{2m+2}(z)$ as
$\sum_{n\gg-\infty}a_{n}(\kappa)\mathbf{e}(nz/h)$ as above. The analysis of
$B_{m}\big(\cosh d(z,w)\big)$ appearing in the proof of Proposition
\ref{inteval} shows that the integral over $u$ converges absolutely for every
$v$, and as the non-principal part of $g$ decays exponentially with $v$, we find
that the integral involving just the part
$\sum_{n=1}^{\infty}a_{n}(\kappa)\mathbf{e}(nz/h)$ of
$g[A_{\kappa}^{-1}]_{2m+2}(z)$ converges absolutely for $s=0$. As we multiply by
$s$ and take the constant term at $s=0$, this part contributes nothing to the
expression in question. As for the (finitely many) other terms, a similar
argument shows that for large enough $\Re s$ the total integral converges
absolutely, hence we may evaluate it in any order we find convenient. We carry
out the integral with respect to $u$ first. By Proposition \ref{inteval} we get
an expression of the sort \[\frac{-4i|r|^{m/2}m!s}{(-8\pi
i)^{m+1}}\int_{M}^{\infty}\bigg(\frac{\sum_{n\geq0}a_{-n}(\kappa)C_{n}}{v}
+\Lambda(v)\bigg)e^{-sv}dv\] (this is a finite sum, since only finitely many
coefficients $a_{-n}(\kappa)$ may not vanish), where $\Lambda$ is a smooth
function of $v$ satisfying $\Lambda(v)=O\big(\frac{1}{v^{2}}\big)$. The integral
involving $\Lambda$ converges also for $s=0$, hence does not contribute to the
final result by the same argument from above. The remaining term is some
constant times the constant term at $s=0$ of the expression
\[\int_{M}^{\infty}\frac{se^{-vs}}{v}dv,\quad\mathrm{which\
equals}\quad\frac{-e^{-vs}}{v}\bigg|_{M}^{\infty}-\int_{M}^{\infty}\frac{e^{-vs}
}{v^{2}}dv=\frac{e^{-Ms}}{M}-\int_{M}^{\infty} \frac{e^{-vs}}{v^{2}}dv\] by
integration by parts. As the latter integral converges also for $s=0$, we may
just substitute this value and obtain
$\frac{1}{M}-\int_{M}^{\infty}\frac{dv}{v^{2}}=0$. Hence the remaining term of
the integral in question also vanishes, which completes the proof of the
theorem.
\end{proof}

\section{Contributions of Poles \label{PolEval}}

It remains to evaluate the contribution of each pole $\tilde{w} \neq w_{0}$ of
$g$, which we write again as $w=\sigma+it$, to the pairing $\big\langle
g,\frac{1}{2\pi i}\delta_{2m}\Phi_{m,r,\beta}^{L}\big\rangle^{reg}$ given in the
form appearing in Proposition \ref{unfw0}. For this we first consider the
function $B_{m}\big(\cosh d(z,w_{0})\big)$ around $z=w \neq w_{0}$. More
precisely, we substitute $z=A_{w}^{-1}(\zeta)$ for $\zeta\in\mathcal{B}_{1}$, so that $\cosh d(z,w_{0})$ takes, by Lemma \ref{b-=1neg} and part $(iv)$ of Lemma \ref{Awprop}, the form
\[1+\frac{|w-w_{0}-(\overline{w}-w_{0})\zeta|^{2}}{2t_{0}t(1-|\zeta|^{2})},\]
and using this we obtain a Taylor expansion of the sort \[B_{m}\big(\cosh
d(z,w_{0})\big)=\sum_{p=0}^{\infty}\sum_{q=0}^{\infty}\alpha_{p,q}^{(m)}(w,w_{0}
)\zeta^{p}\overline{\zeta}^{q}.\] This expansion converges on some ball
$\mathcal{B}_{\delta}$ of positive radius $\delta$ (in fact, we can take
$\delta=|A_{w}(w_{0})|$). However, in order to avoid convergence issues below we
shall fix some $d\geq0$ and take the sum only on $p+q \leq d$, knowing that the
remainder, which we denote $B_{m}^{d}(w,w_{0},\zeta)$, is of growth order
$O\big(|\zeta|^{d+1}\big)$ as $\zeta\to0$.

In addition, we expand $g(z)$ as in Equation \eqref{expwzeta} once again.
Multiplying the expression $(1-\zeta)^{2m}$ appearing there by
$(z-w_{0})^{m}(z-\overline{w}_{0})^{m}$ yields the $m$th power of
\begin{equation}
\psi(w,w_{0},\zeta)=\big(w-w_{0}-(\overline{w}-w_{0})\zeta\big)\big(w-\overline{w}_{0}-(\overline{w}-\overline{w}_{0})\zeta\big). \label{psidef}
\end{equation}
We define, for $n\in\mathbb{Z}$, the function $c_{n}^{(m)}(w,w_{0})$ according
to the Laurent expansion
\[\sum_{n\gg-\infty}c_{n}^{(m)}(w,w_{0})\zeta^{n}=|r|^{\frac{m}{2}}\frac{\psi(w,w_{0},\zeta)^{m}}{(2it_{0}t)^{m}}\sum_{n\gg-\infty}a_{n}(w)\zeta^{n}.\] The examination of the contribution of the pole of $g$ at $w \neq w_{0}$ now begins with
\begin{prop}
If $\tilde{w}=w \neq w_{0}$ is a pole of $g$ of order $d$ then the integral over
$\mathcal{D}_{\varepsilon,w}$ appearing in Proposition \ref{unfw0} defines a
holomorphic function of $s$ with $\Re s>d-1$. Multiplying by the coefficient
from that proposition, we obtain a function admitting an analytic continuation
to the point $s=0$, where it attains the value \[\frac{m!}{2i(-8\pi
i)^{m}|\Gamma_{L,w_{0}}|}\sum_{p}c_{-1-p}^{(m)}(w,w_{0})\alpha_{p,0}^{(m)}(w,w_{
0}).\] \label{contwneqw0}
\end{prop}

\begin{proof}
As in the proof of Proposition \ref{contw0}, we expand $g(z)$ as in Equation
\eqref{expwzeta}, change the variable to $\zeta=A_{\tilde{w}}(z)$, and apply
Equation \eqref{coshd} and part $(iv)$ of Lemma \ref{Awprop}. The resulting
integral becomes, after cancelations,
\[\int_{\mathcal{B}_{\varepsilon}}\frac{\psi(w,w_{0},\zeta)^{m}}{-2i(2it_{0}t)^{
m}}\sum_{n}a_{n}(w)\zeta^{n+1}B_{m}\bigg(1+\frac{\big|w-w_{0}-(\overline{w}-w_{0
})\zeta\big|^{2}}{2t_{0}t(1-|\zeta|^{2})}\bigg)|\zeta|^{s-2}d\zeta
d\overline{\zeta}.\] As the term with $B_{m}$ is bounded on $\mathcal{B}_{\varepsilon}$ and $a_{n}(w)=0$ for $n<-d$, the integral indeed converges wherever $\Re s>d-1$, proving the first assertion. We now multiply by the coefficient $\frac{-2|r|^{m/2}m!s}{(-8\pi i)^{m+1}|\Gamma_{L,w_{0}}|}$ again, and plug in the definition of the Laurent series $\sum_{n}c_{n}^{(m)}(w,w_{0})\zeta^{n+1}$ (note the shift in the power of $\zeta$ appearing already in the last formula). In addition, we decompose the function $B_{m}$ as its Taylor polynomial of total degree $d$ plus the remainder $B_{m}^{d}(w,w_{0},\zeta)$. The estimate on $B_{m}^{d}(w,w_{0},\zeta)$ as $\zeta\to0$ shows that the term of the integrand involving this remainder is bounded on $\mathcal{B}_{\varepsilon}$ also for $s=0$. As we have $s$ in the external coefficient and we are interested in the constant term at $s=0$, this part of the integrand does not contribute to the final result. For the same reason we may also take the sum over $n$ to include just non-positive $n$.

It therefore remains to determine the constant term at $s=0$ of the analytic
continuation of \[\frac{-im!s}{(-8\pi
i)^{m+1}|\Gamma_{L,w_{0}}|}\int_{\mathcal{B}_{\varepsilon}}\sum_{n=-d}^{0}c_{n}^{(m)}(w,w_{0})\sum_{p+q  \leq d}\alpha_{p,q}^{(m)}(w,w_{0})\zeta^{n+1+p}\overline{\zeta}^{q}|\zeta|^{s-2}d\zeta d\overline{\zeta}.\] Writing $\zeta=\rho e^{i\varphi}$ and $d\zeta d\overline{\zeta}=-2i\rho d\rho d\varphi$ once again, this integral takes the
form \[\frac{-2m!s}{(-8\pi i)^{m+1}|\Gamma_{L,w_{0}}|}\!\int_{0}^{\varepsilon}\!\!\!\int_{0}^{2\pi}\!\!\!\sum_{n,p,q}\!\!c_{n}^{(m)}
(w,w_{0})\alpha_{p,q}^{(m)}(w,w_{0})\rho^{n+p+q+s}e^{i(n+1+p-q)\varphi}d\rho d\varphi.\] The integration with respect to $\varphi$ leaves only the terms with $q=n+1+p$, and after carrying out the integration with respect to $\rho$ as well (this is allowed if $\Re s>d-1$) we obtain \[\frac{m!s}{2i(-8\pi
i)^{m}|\Gamma_{L,w_{0}}|}\sum_{n,p}c_{n}^{(m)}(w,w_{0})\alpha_{p,n+1+p}^{(m)}(w,w_{0})\frac{\varepsilon^{s+2n+2p+2}}{s+2n+2p+2}.\] Substituting $s=0$ annihilates all the terms in which $n+p+1\neq0$. In the remaining terms, $s$ is canceled in the fraction, and the power of $\varepsilon$ becomes $\varepsilon^{0}=1$ after the substitution $s=0$. This completes the proof of the proposition.
\end{proof}

\smallskip

A deeper analysis of the coefficients $\alpha_{p,0}(w,w_{0})$ yields a more
succinct formula for the contribution of the pole at $w$. For this we prove
\begin{lem}
The function $\alpha_{p,0}^{(m)}(w,w_{0})$ equals just
\[\frac{(-1)^{p}B_{m}^{(p)}\big(\cosh
d(w,w_{0})\big)(\overline{w}-w_{0})^{p}(\overline{w}-\overline{w}_{0})^{p}}{
p!(2t_{0}t)^{p}},\] where $B_{m}^{(p)}$ is the $p$th derivative of $B_{m}$.
\label{alphap0exp}
\end{lem}

\begin{proof}
The usual Taylor expansion gives us
\[B_{m}\big(\cosh
d(z,w_{0})\big)=\sum_{l=0}^{\infty}\frac{B_{m}^{(l)}\big(\cosh
d(w,w_{0})\big)}{l!}\cdot\big[\cosh d(z,w_{0})-\cosh d(w,w_{0})\big]^{l}.\] We
have seen in the proof of Proposition \ref{contwneqw0} that \[\cosh
d(z,w_{0})=1+\frac{\big|w-w_{0}-(\overline{w}-w_{0})\zeta\big|^{2}}{2t_{0}
t(1-|\zeta|^{2})}\] in terms of $\zeta$, and $\cosh d(w,w_{0})$ is the same
expression but with $\zeta=0$. We write $\frac{1}{1-|\zeta|^{2}}$ as
$\sum_{n=0}^{\infty}|\zeta|^{2n}$, and expanding the absolute value appearing in
the numerator we get a power series in $\zeta$ and $\overline{\zeta}$. Now, we
are interested only in the coefficients $\alpha_{p,0}^{(m)}$. Hence we may omit
all the terms involving $\overline{\zeta}$, in particular those which are
multiplied by some positive power of $|\zeta|^{2}$. This allows us to ignore the
denominator $1-|\zeta|^{2}$ in $\cosh d(z,w_{0})$. The difference from $\cosh
d(w,w_{0})$ then takes the form
\[\frac{|\overline{w}-w_{0}|^{2}|\zeta|^{2}-(w-w_{0})(w-\overline{w}_{0}
)\overline{\zeta}-(\overline{w}-w_{0})(\overline{w}-\overline{w}_{0})\zeta}{2t_{
0}t}.\] Once again, the first two terms do not contribute to any of the
coefficients $\alpha_{p,0}^{(m)}$, and the $p$th power of the remaining term
gives us the asserted value for $\alpha_{p,0}^{(m)}$. This proves the lemma.
\end{proof}

\smallskip

We are now in place to prove the final formula for the regularized pairing of
the meromorphic modular form $\frac{1}{2\pi i}\delta_{2m}\Phi_{m,r,\beta}^{L}$
from Theorem \ref{delta2mPhi} with any meromorphic modular form $g$ of weight
$2m+2$. To do this we define for two distinct points $w=\sigma+it$ and
$w_{0}=\sigma_{0}+it_{0}$ in $\mathcal{H}$ the radius
$\delta=\big|A_{w}(w_{0})\big|>0$, and given such a modular form $g$ we let
$\Psi_{g,w,w_{0}}^{(m)}:\mathcal{B}_{\delta}\to\mathbb{C}$ be the (meromorphic)
function in which $\Psi_{g,w,w_{0}}^{(m)}(\zeta)$ equals
\begin{equation}
g[A_{w}^{-1}]_{2m+2}(\zeta)\frac{\psi(w,w_{0},\zeta)^{m}}{(2it_{0}t)^{m}}B_{m}\bigg(\cosh
d(w,w_{0})-\frac{(\overline{w}-w_{0})(\overline{w}-\overline{w}_{0})}{2t_{0}t}\zeta\bigg),
\label{Psidef}
\end{equation}
where $\psi(w,w_{0},\zeta)$ is the expression defined in Equation \eqref{psidef} (recall that Lemma \ref{singB} extends $B_{m}$ to a holomorphic function of  $T\in\mathbb{C}$ with $\Re T>1$, and $\delta$ is the radius making sure that the argument of $B_{m}$ remains in this domain). Our final formula is now given in
\begin{thm}
Let $g$ be a meromorphic modular form of weight $2m+2$ with respect to $\Gamma$,
and let $\big\{w_{j}^{\pm}=\sigma_{j}^{\pm}+it_{j}^{\pm}\big\}_{j=1}^{l_{\pm}}$
be representatives for the $\Gamma$-orbits forming the set $S_{\pm\beta,r}$ defined before Corollary \ref{b-=1lift}. Then the regularized pairing $\big\langle g,\frac{1}{2\pi
i}\delta_{2m}\Phi_{m,r,\beta}^{L} \big\rangle^{reg}$ equals \[\frac{m!}{2i(-8\pi
i)^{m}}\sum_{j,\pm}\frac{(\pm)^{m}}{\big|\Gamma_{L,w_{j}^{\pm}}\big|}\sum_{g(w)=\infty,\ w \neq w_{j}^{\pm}}|r|^{\frac{m}{2}}\mathrm{Res}_{\zeta=0}\big(\Psi_{g,w,w_{j}^{\pm}}^{(m)}(\zeta)d\zeta\big).\] Here the inner sum is over the poles $w$ of $g$ (apart from $w_{j}^{\pm}$ in case it is also a pole), $\Psi_{g,w,w_{j}^{\pm}}^{(m)}$ is defined in Equation
\eqref{Psidef}, and the residue at $\zeta=0$ can also be written as the residue at $z=w$ of
\[g(z)\frac{(z-w_{j}^{\pm})^{m}(z-\overline{w}_{j}^{\pm})^{m}}{(t_{j}^{\pm})^{m}}B_{m}\bigg(\cosh
d(w,w_{0})-\frac{(\overline{w}-w_{0})(\overline{w}-\overline{w}_{0})}{2t_{0}t}A_{w}(z)\bigg)dz,\]
where $B_{m}$ is defined in Lemma \ref{Bm1/T2}, $A_{w}$ is defined in Section \ref{Expatw}, and
$d$ is the hyperbolic distance. \label{final}
\end{thm}

\begin{proof}
We consider the contribution obtained from each representative $w_{j}^{\pm}$,
recalling that summands arising from $w_{j}^{-}$ come with the sign $(-1)^{m}$.
Proposition \ref{contw0} shows that when we consider the poles of $g$ we may
ignore the pole in $w_{j}^{\pm}$ itself, in case such a pole exists. Moreover,
Theorem \ref{cusp0} allows us to ignore the regularized integrals arising from
the cusps. It just remains to apply Proposition \ref{contwneqw0} with
$w_{0}=w_{j}^{\pm}$, and find that the contribution from each pole $w$ of $g$
is \[\frac{m!}{2i(-8\pi
i)^{m}}\sum_{j,\pm}\frac{(\pm)^{m}}{\big|\Gamma_{L,w_{j}^{\pm}}\big|}\sum_{p}c_{-1-p}^{(m)}(w,w_{j}^{\pm})\alpha_{p,0}^{(m)}(w,w_{j}^{\pm}).\] But Lemma \ref{alphap0exp} shows that the numbers $\alpha_{p,0}^{(m)}(w,w_{j}^{\pm})$ are the coefficients of the expansion of the (holomorphic) function involving $B_{m}$ around $\zeta=0$. Hence the sum over $p$ is just the $-1$st coefficient of the expansion of $\Psi_{g,w,w_{j}^{\pm}}^{(m)}$ around $\zeta=0$, which is the asserted residue. The usual change of variables $\zeta=A_{w}(z)$ and
$z=A_{w}^{-1}(\zeta)$, together with the calculations we did in the proof of
Proposition \ref{contwneqw0}, transform the residue of
$\Psi_{g,w,w_{j}^{\pm}}^{(m)}(\zeta)d\zeta$ at $\zeta=0$ to the asserted residue
at $z=w$. This completes the proof of the theorem.
\end{proof}

\section{Lattices for Integral Quadratic Forms \label{QuadForms}}

Let $N$ be a positive integer, and let $\beta$ be an element of
$\mathbb{Z}/2N\mathbb{Z}$. Consider the set of integral binary quadratic forms
$Q(X,Y)=AX^{2}+BXY+CY^{2}$, of discriminant $D=B^{2}-4AC$, such that $A$ is
positive and divisible by $N$, and $B$ lies in $\beta+2N\mathbb{Z}$. The group
$\Gamma_{0}(N)$ of matrices $\binom{a\ \ b}{c\ \ d} \in SL_{2}(\mathbb{Z})$ in
which $N|c$ preserves this set under the action in which
$\gamma(Q)(X,Y)=Q\big((X,Y)\binom{0\ \ 1}{1\ \ 0}\gamma\binom{0\ \ 1}{1\ \
0}\big)$. The following relation to lattices is well-known and easy to prove:
\begin{lem}
$(i)$ Identify the quadratic form $Q(X,Y)=AX^{2}+BXY+CY^{2}$ in which $N|A$ and
$2N|B$ with the matrix $\lambda=\binom{B/2\sqrt{N}\ \ \ \ C/\sqrt{N}\
}{-A/\sqrt{N}\ \ -B/2\sqrt{N}}$. The images of these quadratic forms form a
lattice $L$ in $M_{2}(\mathbb{R})$, in which $\lambda^{2}=\frac{D}{2N}$. $(ii)$
The dual lattice $L^{*}$ corresponds to those quadratic forms in which $N|A$ but
$B\in\mathbb{Z}$ is arbitrary. The discriminant group $D_{L}$ is
$\mathbb{Z}/2N\mathbb{Z}$ (the projection from $L^{*}$ just takes the class of
$B$), with $\frac{\gamma^{2}}{2}$ being the image of $\frac{B^{2}}{4N}$ in
$\mathbb{Q}/\mathbb{Z}$. $(iii)$ The action of $\gamma\in\Gamma_{0}(N)$ on $Q$
described above corresponds to its action on $\lambda$ by conjugation. This
identifies the quotient $\Gamma_{0}(N)/\{\pm I\}$ with a the discriminant kernel
of $L$. \label{quadlat}
\end{lem}

\begin{proof}
The only part here which is not straightforward is the assertion that
$\Gamma_{0}(N)/\{\pm I\}$ surjects onto the discriminant kernel of $L$. But this
statement appears in Proposition 2.2 of \cite{[BO]}. This proves the lemma.
\end{proof}

Since the lattice $L$ from Lemma \ref{quadlat} is isotropic, it is conventional
to take the isotropic vector $z \in L_{\mathbb{R}}$ which is used for the
definition of $K_{\mathbb{R}}$ to be a primitive element of $L$. Hence we
replace the previous vector $z$ by its multiple $\binom{0\ \ 1/\sqrt{N}}{0\ \ \
\ 0\ \ \ }$, so that the complementary vector $\zeta$ is taken to be the
isotropic vector $\binom{\ \ 0\ \ \ 0}{\sqrt{N}\ \ 0}$. The lattice
$K=(z^{\perp} \cap L)/\mathbb{Z}z$ is spanned by $\binom{\sqrt{N}\ \ \ \ 0\ \ }{\ \ 0\ \ -\sqrt{N}}$, and dividing this generator by $2N$ yields a generator for $K^{*}$. Hence $D_{K}=D_{L}$. The identification of $K_{\mathbb{R}}+iC$ with $\mathcal{H}$ takes $z\in\mathcal{H}$ to $z\binom{\sqrt{N}\ \ \ \ 0\ \ }{\ \ 0\ \ -\sqrt{N}}$ (with vector norm $2Nz^{2}$), so that $Z_{V,Z}$ is just
$\sqrt{N}M_{z}$, and the associated negative definite part is still spanned by
$J_{z}$. Combining these results with Lemma \ref{b-=1exp} now proves
\begin{lem}
The pairing of the vector $\lambda$ associated with $Q$ with $Z_{V,Z}$ gives
$Az^{2}+Bz+C=Q(z,1)$ in the notation of \cite{[BK]}. Pairing the former vector
with $\sqrt{N}J_{z}$ (of vector norm $-2N$) gives $\frac{A|z|^{2}+Bu+C}{v}$,
which is denoted by $Q_{z}$ in that reference. \label{quadexp}
\end{lem}

\smallskip

Given $N$ and $\beta$ as above as well as a negative discriminant $D$, we define $\mathcal{Q}_{\beta,D}^{N}$ to be the set of integral binary quadratic forms as
above (with $A>0$), whose discriminant equals $D$. They are all positive definite. By part $(iii)$ of Lemma \ref{quadlat} these quadratic forms form an orbit of $\Gamma_{0}(N)$ (or perhaps the union of finitely many orbits), and part $(i)$ of that Lemma shows that the parameter $r$ we used above equals $\frac{D}{4N}$. The non-triviality relation $r\in\frac{\beta^{2}}{2}+\mathbb{Z}$ is the usual condition $D \equiv B^{2}\pmod{4N}$, a condition which we assume from now on. The quadratic forms satisfying these conditions but in which $A$ is negative (i.e., those which are negative definite) are the additive inverses of the quadratic forms from $\mathcal{Q}_{-\beta,D}^{N}$.

Generalizing the meromorphic modular forms defined in \cite{[BK]} to level $N$,
we define the weight $2m+2$ meromorphic modular form
\[f_{m+1,\beta,D}(z)=\frac{|D|^{\frac{m+1}{2}}}{2N^{\frac{m}{2}}}
\sum_{Q\in\mathcal{Q}_{\beta,D}^{N}\cup(-1)^{m}\mathcal{Q}_{-\beta,D}^{N}}\frac{1}{Q(z,1)^{m+1}},\] where the union with $(-1)^{m+1}\mathcal{Q}_{-\beta,D}^{N}$ has the same meaning as the union with $(-1)^{m}S_{-\beta,r}$ above. In addition, we consider the function \[\mathcal{F}_{\beta,D,-1}=\sum_{Q\in\mathcal{Q}_{\beta,D}^{N}\cup(-1)^{m}\mathcal{Q}_{-\beta,D}^{N}}
\frac{Q(z,1)^{m}}{2(N|D|)^{\frac{m}{2}}}\int_{0}^{\mathrm{arctanh}(\sqrt{|D|}/Q_{z})}\sinh^{2m}\theta d\theta.\] Note that if $2\beta=0$ (hence with even $m$) there are no cancelations, since $Q$ always stands for a positive definite quadratic form. This is always the case if $N=1$. We now prove
\begin{prop}
Our $\mathcal{F}_{\beta,D,-1}$ generalizes the function denoted by
$\mathcal{F}_{Q,-1}$ in \cite{[BK]} to the case of level $N$. \label{F-1gen}
\end{prop}

\begin{proof}
A quadratic form $Q\in\mathcal{Q}_{\beta,D}^{N}$ was seen to correspond to
$-\sqrt{|r|}J_{w}$, where $r=\frac{D}{4N}$ and $w=\sigma+it$ is the unique
element of $\mathcal{H}$ satisfying $Q(w,1)=0$ (this point $w$ was denoted by
$z_{Q}$ in \cite{[BK]}). The entry denoted by $a$ in Lemma \ref{b-=1exp} equals
$\frac{A}{\sqrt{N}}$ in Lemma \ref{quadexp} as well as $\frac{\sqrt{|r|}}{t}$
in  Lemma \ref{b-=1neg}, so that we obtain from the value of $r$ and from
Equation \eqref{coshd} the equalities
\begin{equation}
Q(z,1)=\sqrt{|D|}\cdot\frac{(z-w)(z-\overline{w})}{2t}\quad\mathrm{and}\quad
Q_{z}=\sqrt{|D|}\cosh d(z,w). \label{Qz1D}
\end{equation}
Now, the coefficient $c_{-1-n,Q}(z)=c_{0,Q}(z)$ appearing in Equation (8.3) of
\cite{[BK]} is just the constant term in the expansion of
$\frac{(\tau-w)^{m}(\tau-\overline{w})^{m}}{(2t)^{m}(\tau-\overline{z})^{2m+1}}$
around $\tau=z$ (this is $G_{z,w,1}(\tau)$ in the notation of \cite{[BK]}, with
$k=m+1$). It can be evaluated by a simple substitution $\tau=z$, yielding the
value $\frac{Q(z,1)^{m}}{|D|^{m/2}(2iv)^{2m+1}}$ by Equation \eqref{Qz1D}. If
$2\beta=0$ (like when $N=1$) then the union
$\mathcal{Q}_{\beta,D}^{N}\cup\mathcal{Q}_{-\beta,D}^{N}$ (recall that $m$ is
even) reduces to one set $\mathcal{Q}_{\beta,D}^{N}$, but the factor 2 in the
denominator is canceled. This proves the proposition.
\end{proof}

\smallskip

We are now able to establish the relation between our theta lifts and the
modular forms from \cite{[BK]}:
\begin{prop}
The function
$\frac{(4v^{2}\partial_{\overline{z}})^{2m}\Phi_{m,r,\beta}^{L}}{(2m)!}$ from
Proposition \ref{D2m+1xi} becomes, for the lattice $L$ defined in
Lemma \ref{quadlat} and with $r=\frac{D}{4N}$, the function
$\mathcal{F}_{\beta,D,-1}$ multiplied by $\frac{4m!|D|^{m/2}}{(-4\pi i)^{m}}$.
The function from Theorem \ref{delta2mPhi} equals, in this case,
$-\frac{|D|^{m/2}m!}{(8i)^{m}\pi^{m+1}}$ times the modular form
$f_{m+1,\beta,D}$. \label{reltoBK}
\end{prop}

\begin{proof}
The argument leading to Equation \eqref{Qz1D} also shows that the set of points
$w\in\mathcal{H}$ such that $Q(w,1)$ vanishes for some
$Q\in\mathcal{Q}_{\beta,D}^{N}$ is precisely the set denoted $S_{\beta,r}$ (with
$r=\frac{D}{4N}$) above. Substituting this relation, the value of $r$, and
Equation \eqref{Qz1D} into the expression from Corollary \ref{b-=1lift} shows
that if $z$ does not lie in $S_{\beta,r} \cup S_{-\beta,r}$ then
$\Phi_{m,D/4N,\beta}^{L}(z)$ equals \[\frac{m!}{(8\sqrt{N}\pi
i)^{m}}\sum_{Q\in\mathcal{Q}_{\beta,D}^{N}\cup(-1)^{m}\mathcal{Q}_{-\beta,D}^{N}
}\frac{Q(\overline{z},1)^{m}}{v^{2m}}B\bigg(m+\frac{1}{2},-m;\frac{|D|}{Q_{z
}^{2}}\bigg)\] in this case. Now, the proof of Proposition \ref{D2m+1xi} shows
that first function in question is $(2v^{2})^{m}$ times the complex conjugate of
the latter expression. In addition, the change of variable
$\theta=\mathrm{arctanh}\sqrt{\xi}$ and $\xi=\tanh^{2}\theta$ yields the
equalities
\[\sinh^{2m}(\theta)=\bigg(\frac{\tanh^{2}\theta}{1-\tanh^{2}\theta}\bigg)^{m}
=\frac{\xi^{m}}{(1-\xi)^{m}}\quad\mathrm{and}\quad
d\theta=\frac{d\xi}{2\sqrt{\xi}(1-\xi)}.\] Hence the integral appearing in the
definition of $\mathcal{F}_{\beta,D,-1}$ is just the incomplete beta function
$\frac{1}{2}B\Big(m+\frac{1}{2},-m;\frac{|D|}{Q_{z}^{2}}\Big)$. This
proves the first relation. Plugging the value of $r$ and the expression from
Equation \eqref{Qz1D} into the formula for $\frac{1}{2\pi
i}\delta_{2m}\Phi_{m,D/4N,\beta}^{L}$ from Theorem \ref{delta2mPhi} yields the
function
\[-\frac{4|D|^{m/2}m!}{(i\sqrt{N})^{m}(8\pi)^{m+1}}\sum_{Q\in\mathcal{Q}_{\beta,
D}^{N}\cup(-1)^{m}\mathcal{Q}_{-\beta,D}^{N}}\frac{\sqrt{|D|}^{m+1}}{Q(z,1)^{m+1
}},\] which is easily seen to be $f_{m+1,\beta,D}$ times the asserted constant.
This completes the proof of the proposition.
\end{proof}

The same argument as in the proof of Proposition \ref{reltoBK} shows that for
this lattice, the expression given in Theorem \ref{final} for the pairing
$\big\langle g,\frac{1}{2\pi i}\delta_{2m}\Phi_{m,r,\beta}^{L}\big\rangle^{reg}$
is the constant $\frac{m!/2i}{(-8\sqrt{N}\pi i)^{m}}$ times
\[\sum_{j,\pm}\frac{(\pm)^{m}}{\big|\Gamma_{Q_{j}^{\pm}}\big|}\!\sum_{\substack{
g(w)=\infty \\
Q_{j}^{\pm}(w,1)\neq0}}\!\mathrm{Res}_{z=w}\bigg[g(z)Q_{j}^{\pm}(z,1)^{m}B_{m}
\bigg(\frac{(Q_{j}^{\pm})_{z}}{\sqrt{|D|}}-\frac{Q_{j}^{\pm}(z,1)^{m}}{\sqrt{|D|
}t}A_{w}(z)\bigg)dz\bigg],\] where the $Q_{j}^{\pm}$ are representatives for
the sets $\mathcal{Q}_{\pm\beta,D}^{N}$ modulo the action of $\Gamma_{0}(N)$.
Note that by taking only $p=0$ in Proposition \ref{contwneqw0} (namely replacing
the function $B_{m}$ with its value at $z=w$) we obtain the required constant
from Proposition \ref{reltoBK} (with $N=1$) times the value of the pairing given
in Theorem 1.1 of \cite{[BK]}. Indeed, our incomplete beta function is twice the
integral over $\theta$ appearing in that reference, and $w_{0}=z_{Q}$ (or $Q$)
is counted there twice, once as an element of $\mathcal{Q}_{\beta,D}^{1}$ and once
as an element of $\mathcal{Q}_{-\beta,D}^{1}$.

\smallskip

We conclude with some remarks about the geometric context of the constructions
in this paper. In the case presented in this section, as well as the more
general case in which the group $\Gamma$ is related to indefinite rational
quaternion algebras (as in Section 1 of \cite{[Ze3]}), the curve $Y_{\Gamma}$
serves as the moduli space of elliptic curves, or Abelian surfaces with
quaternion multiplication, with some additional data. Hence $Y_{\Gamma}$ (as
well as $X_{\Gamma}$) carries universal families of symmetric powers of such
objects, yielding local systems of the sort described in \cite{[Ze3]}. Modular
forms with the associated representations are investigated in detail in
\cite{[Ze2]}, and some components of the latter functions can be interpreted as
elements of cohomology groups of these universal families. Indeed, our function
$\frac{1}{2\pi i}\delta_{2m}\Phi_{m,r,\beta}^{L}$ may be completed to such a
meromorphic vector-valued differential form, admitting a vector-valued pre-image
under $\partial$. This pre-image contains $\Phi_{m,r,\beta}^{L}$ as its weight
$2m$ component, as well as the function from Proposition \ref{D2m+1xi} as the
weight $-2m$ component. In fact, one can also prove, using associated Legendre
polynomials and certain differential equations, that the weight 0 component of
this pre-image is $(-1)^{m+1}\frac{|r|^{m/2}m!}{(2\pi)^{m}}$ times a real
function with certain logarithmic singularities on $X_{\Gamma}$. In the case
considered in this section the latter function is the specialization of the
global higher Green's function $G^{\Gamma\backslash\mathcal{H}}_{m+1}$ of
\cite{[Me]} in which one variable is taken from the (finite) image of
$S_{\beta,r}\cup(-1)^{m}S_{-\beta,r}$ (interpreted as usual) in $X_{\Gamma}$.
In the more general setting it is again a function of the same type, but in
which the summation of the local higher Green's function $G^{\mathcal{H}}_{m+1}$
is carried over a different group $\Gamma$. Hence the vector-valued pre-image
mentioned here, which we may construct by applying powers of the weight lowering
operator $4v^{2}\partial_{\overline{z}}$ on $\Phi_{m,r,\beta}^{L}$, coincides
(in the special case considered in this section) with the vector-valued
functions appearing in \cite{[Me]}.

\noindent\textsc{Einstein Institute of Mathematics, the Hebrew University of Jerusalem, Edmund Safra Campus, Jerusalem 91904, Israel}

\noindent E-mail address: zemels@math.huji.ac.il

\end{document}